\documentclass[11pt]{amsart}
\usepackage[bitstream-charter]{mathdesign}
\usepackage{geometry}

\usepackage{amsmath}
\usepackage{amsthm, amssymb, graphicx,bm}
\usepackage[mathcal]{eucal} 

\usepackage[pdftex,bookmarks=true]{hyperref}
\usepackage{color}
\usepackage{comment}
\usepackage{tikz}

\newtheorem{theorem}{Theorem}
\newtheorem*{theorem*}{Theorem}

\newtheorem{corollary}{Corollary}
\newtheorem*{corollary*}{Corollary}
\newtheorem{definition}{Definition}
\newtheorem*{definition*}{Definition}
\newtheorem{lemma}{Lemma}
\newtheorem*{lemma*}{Lemma}

\newtheorem{proposition}{Proposition}
\newtheorem*{proposition*}{Proposition}

\numberwithin{equation}{section}
\newtheorem{remark}{Remark}

\newtheorem*{condition*}{Condition}
\newtheorem{claim}{Claim}

\def\supp{{\text{supp }}}
\def\C{{\mathbb C}}
\def\R{{\mathbb R}}

\def\E{{\mathbb E}}
\def\Nat{{\mathbb N}}

\def\<{{\langle}}
\def\>{{\rangle}}

\def\s{{\mathcal{S}}} 
\def\A{{\mathbb A}}
\def\Z{{\mathbb Z}}
\def\L{{\mathcal L}}
 
\title[Embeddings from weighted $L^p$ into outer measure spaces]{Generalized Carleson Embeddings  into Weighted Outer Measure Spaces}

\author{Yen Do}
\author{Mark Lewers}
\address{Department of Mathematics, The University of Virginia, Charlottesville, VA 22904-4137}
\email{yendo@virginia.edu}
\email{mel3jp@virginia.edu}

\thanks{Y.D. partially supported by NSF grant DMS-1800855.}
\date{\today}

\begin{document}

\begin{abstract} 
We prove generalized Carleson embeddings for the continuous wave packet transform from $L^p(\R,w)$ into an outer $L^p$ space over $\R \times \R \times (0,\infty)$ for $2< p < \infty$ and weight $w \in \A_{p/2}$. This work is  a weighted extension of the corresponding Lebesgue result in \cite{dt2015} and generalizes a similar result in \cite{dl2012sm}. The proof in this article relies on $L^2$ restriction estimates for the wave packet transform  which are geometric and may be of  independent interest. 
\end{abstract}

\maketitle

\section{Introduction}

Lennart Carleson's influential paper \cite{carleson1966} in 1966 resolved Lusin's Conjecture by proving the Fourier series of a function $f \in L^2[0,1]$ converges almost everywhere; see also the work of Hunt \cite{hunt1968}. The techniques used in Carleson's proof, now referred to as \textit{time-frequency analysis}, have since played an important role in analysis and serve as a tool in proving $L^p$ estimates on modulation invariant integral operators. We highlight for example Fefferman's proof of Lusin's Conjecture in \cite{fefferman1973} along with  Lacey and Thiele's use of time-frequency analysis in their work on the bilinear Hilbert transform \cite{lt1997,lt1999} and the Carleson operator \cite{lt2000}.

The methods of time-frequency analysis usually pass the analysis on a multilinear form $\Lambda$ to a model sum
    \begin{equation}\label{e.modelsumP}
    \Lambda(f_1,...,f_n) \sim \sum_{s} c_s(\Lambda) \cdot a_{s,1}(f_1) \cdots a_{s,n}(f_n) 
    \end{equation}
indexed over a discrete collection of rectangles $s$ in the phase plane. The terms within each summand, localized to $s$, are either dependent on the form $\Lambda$ or one of the input functions $f_j$. From there, the rectangles are grouped into specified collections on which the desired estimates are obtainable. As shown by the first author and Thiele in \cite{dt2015}, this procedure follows an  outer measure framework which  focuses on two main steps in proving $L^p$ estimates for multilinear forms. The first step is to estimate the multilinear form $\Lambda$ by applying H\"older's inequality in the context of outer measures,
    \begin{equation}\label{e.outerHolder_abstract}
     |\Lambda(f_1,...,f_n)| \le C\prod_{j=1}^n \big\|F_j(f_j)\big\|_{\L^{p_j}(X,\sigma,\s_j)}.  
    \end{equation}
Here, $\L^{p_j}(X,\sigma,\s_j)$ is an \textit{outer $L^{p_j}$ space} constructed over a suitable outer measure space $(X,\sigma,\s_j)$.  The formulation of outer $L^p$ spaces will be discussed in  Section~\ref{s.OuterLp}. The operators $F_j$ are akin to the $a_{s,j}$ in \eqref{e.modelsumP} and represent a suitable projection of $\Lambda$ over  $X$.  The final steps are then to establish outer measure $L^p$ embeddings on each operator $F_j$ in the form
    \begin{equation}\label{e.CarlesonEmbedding_abstract}
         \big\| F_j(f_j) \big\|_{\L^{p_j}(X,\sigma,\s_j)} \le C(F_j,p_j) \| f_j \|_{L^{p_j}(\R)} .
    \end{equation}
Estimates such as \eqref{e.CarlesonEmbedding_abstract} are referred to, see \cite{dt2015,dpo2018}, as \textit{generalized Carleson embeddings}.

What the outer measure framework reveals is that the crux in establishing inequalities on modulation invariant operators is passed to proving Carleson embeddings of form \eqref{e.CarlesonEmbedding_abstract}. This is seen in several recent articles which focus on solving certain Carleson embeddings in order to obtain $L^p$ estimates for specific operators. In the original article \cite{dt2015} which introduced the outer measure framework, the key result in reproving  $L^p$ estimates on the bilinear Hilbert transform with the same restricted range as \cite{lt1997} is a Carleson embedding of the wave packet transform, see \eqref{e.3DPoisson}.  Di Plinio and Ou in \cite{dpo2018} later recovered  $L^p$ estimates for the bilinear Hilbert transform in the full range of \cite{lt1999} by proving a localized  Carleson embedding for the same wave packet transform; here localized is in the sense of \eqref{e.PoissonLocal}. We also mention the work by Uraltsev \cite{uraltsev2016} in reproving $L^p$ estimates for the variational Carleson operator, first obtained in \cite{osttw2012}, which relies on Carleson embeddings for a modified wave packet transform in addition to the wave packet transform \eqref{e.3DPoisson}.  Note the Carleson embedding results stated in this paragraph are  done with respect to functions $f$ in Lebesgue $L^p$ and embeddings into a "non-weighted" outer measure space.

The purpose of this paper is to explore the outer measure framework of time-frequency analysis in the context of weighted inequalities. Specifically, we seek to understand generalized Carleson embeddings 
    $$ \big\| F(f) \big\|_{\L^{p}(X,\sigma^w,\s^w)} \le C(p,w) \|f\|_{L^{p}(\R,w)} $$
where $w$ is a weight on $\R$ and $(X,\sigma^w,\s^w)$ is an outer measure space dependent on $w$.  The motivation is in part  due to the recent progress in identifying weighted $L^p$ estimates in harmonic analysis. Using a weighted time-frequency analysis based on the model sum \eqref{e.modelsumP}, the first author with Lacey \cite{dl2012sm,dl2012jfa} obtained novel weighted estimates for the variational Carleson operator and Walsh counterpart. Part of the analysis in the series focused on inequalities in traditional time-frequency analysis analog to a Carleson embedding \eqref{e.CarlesonEmbedding_abstract} for weighted $L^p$ functions. We are interested in understanding the embeddings in the outer measure framework. 

It is worth mentioning there are suitable alternatives for obtaining weighted $L^p$ estimates in time-frequency analysis which have been recently explored. The application of sparse domination techniques for instance has seen success with the highlight being the remarkable find by Culiuc, Di Plinio, and Ou \cite{cdpo2018} in determining weighted estimates for the bilinear Hilbert transform, the first of its kind\footnote{Xiaochun Li \cite{li} has some unpublished results about weighted estimates for the bilinear Hilbert transform.}. We point out more recent work with weighted estimates for the bilinear Hilbert transform and similar operators  by Cruz-Uribe and Martell \cite{cum2018}, and Benea and Muscalu \cite{bm2017,bm2018}.  Note that sparse domination has also been used to obtain  weighted norm inequalities for the variational Carleson operator  \cite{dpdu2018} which are an improvement of \cite{dl2012sm}. It is questionable however if sparse domination can be used to establish weighted norm estimates for operators with less symmetry  such as the truncated bilinear Hilbert transform  \cite{dop2013} or the  biest operator \cite{mtt2004ma1, mtt2004ma2} whose weighted results are unknown.

\subsection{Continuous Wave Packet Transform and Main Result}\label{subs.OuterLp}
The space we work over is upper 3-space $X =  \R \times \R \times \R_+$ whose coordinates are viewed as parameterizations of symmetries on the class of modulation invariant integral operators.  The primary  outer $L^p$ embedding map of interest in this work is the  \textit{wavelet projection operator} of a function $f: \R \to \C$ into upper 3-space
\begin{equation}\label{e.3DPoisson}
P (f)(y,\eta,t) := f * \phi_{\eta,t} (y) = \int_\R f(x) e^{i\eta (y-x)} \frac{1}{t} \phi\Big(\frac{y-x}{t}\Big) \, dx, \qquad (y,\eta,t) \in X 
\end{equation}
where  $\phi_{\eta,t}(y)$  is a modulated wave function of a Schwartz function $\phi$ on $\R$ with compact frequency support. We also refer to the formulation $P(f)(y,\eta,t) = \<f,\phi_{y,\eta,t}\>$ where
\begin{equation}\label{e.L1wavepacket}
\phi_{y,\eta,t}(x) = e^{-i\eta(y-x)} \frac 1 {t} \overline{\phi \Big(\frac{y-x}t \Big)}
\end{equation}
is the \textit{wave packet} of $\phi$ at $(y,\eta,t) \in X$, here as usual $\<f,g\>=\int f \overline{g}$. In this regard, $P(f)$ is also referred to as the \textit{continuous wave packet transform} of $f$. 

The wave packet transform \eqref{e.3DPoisson} serves as a projection of modulation invariant operators in upper 3-space. As shown in \cite{dt2015}, the wave packet representation of the bilinear Hilbert transform in $X$ is a linear combination of integrals whose integrand is a pointwise product of wave packet transforms. To recover $L^p$ estimates for the bilinear Hilbert transform, the key result in \cite[Theorem 5.1]{dt2015} is that the wave packet transform is a generalized Carleson embedding from $L^p(\R)$ to some outer $L^p$ space on $X$ for $2<p<\infty$, 
	\begin{equation}\label{e.PoissonLeb}
	\|P(f)\|_{\mathcal{L}^p(X,\sigma,\s)} \le C_{p,\phi} \, \|f\|_{L^p(\R)} \, .
	\end{equation}  
It was also stated in \cite{dt2015} and later shown in \cite{uraltsev2016} that $P(f)$ is one of two embedding maps arising from the  decomposition \eqref{e.outerHolder_abstract} in connection with the Carleson and variational Carleson operators. 

The main result of this paper  is an extension of \eqref{e.PoissonLeb} to  $\A_p$ weighted spaces. Given $1<p<\infty$, recall a weight $w: \R \to [0,\infty]$ belongs to the class of  $\A_p$ weights if
	\begin{equation}\label{e:ApWeight}
	[w]_{\A_p} := \sup_{I \subset \R}  \left(\frac{1}{|I|} \int_I w(x) \, dx \right)\left(\frac{1}{|I|} \int_I w(x)^{-\frac{1}{p-1}} \, dx\right)^{p-1} < \infty.
	\end{equation}
	
\begin{theorem} \label{t.wembed}
Fix a  Schwartz function $\phi$ on $\R$ whose Fourier transform is supported in a small neighborhood $(-\delta,\delta)$. Let $2<q<\infty$ and $w \in \A_{q/2}$. Then 
\begin{equation}\label{e.wembed}
\| P(f)\|_{\mathcal L^{q}(X,\sigma^w,\s^w)} \lesssim \|f\|_{L^{q}(\R,w)}
\end{equation}
for all $f \in L^q(\R,w)$ where the implicit constant depends on  $\phi$, $q$, and $[w]_{\A_{q/2}}$.  
\end{theorem}

The details concerning the outer $L^p$ space in \eqref{e.wembed} are postponed to Section~\ref{s.OuterLp}. In the scenario $w=1$ is associated with Lebesgue measure,  the theorem immediately implies the strong embedding result \eqref{e.PoissonLeb} from \cite[Theorem 5.1]{dt2015} as $w=1$ is an $\A_p$ weight for all $p > 1$.

We envision \eqref{e.wembed} can be used akin to the Lebesgue version \eqref{e.PoissonLeb} to establish weighted estimates in time-frequency analysis. The weighted results previously mentioned  which use the outer measure framework are based on embeddings which send Lebesgue $L^p$  functions $f$ into non-weighted outer measure spaces. Developing  a weighted outer measure framework in time-frequency analysis\footnote{We point out work by Thiele, Treil, and Volberg \cite{ttv2015} which uses weighted outer measure spaces in the context of martingale multipliers.} could lead to natural self-contained proofs and potentially be a tool to examine the open problems previously mentioned. We stress that using the weighted outer measure framework in this work to obtain weighted $L^p$ estimates for modulation invariant operators is beyond the scope of the paper.
 
Another question not being addressed in this work is the prospect of a  localized version of \eqref{e.wembed}. While the Carleson embedding of $P(f)$  \eqref{e.PoissonLeb} is not bounded for $1 \le p \le 2$, Di Plinio and Ou \cite[Theorem 1]{dpo2018} showed there is a localized extension for $1<p<2$ in the sense
    \begin{equation}\label{e.PoissonLocal}
    \| P(f) 1_{X \backslash E_f} \|_{\mathcal{L}^q(X,\sigma,\s)} \le C_{p,q} \| f\|_{L^p(\R)} \, , \, \hspace{3em} p' < q \le \infty
    \end{equation}
where $E_f$ is an exceptional set dependent on large $L^p$ averages of $f$. Localized embeddings of form \eqref{e.PoissonLocal} are key ingredients in recent papers concerning $L^p$ estimates for modulation invariant operators, cf.\ \cite{uraltsev2016,dpo2018,cdpo2018,dpdu2018}  but it is open whether a localized version of \eqref{e.wembed} holds; this question is left for further study.

\subsection{Structure of Paper} In Section~\ref{s.OuterLp}, we setup the outer measure space over upper 3-space $X$ and the corresponding outer $L^p$ space which is the setting for Theorem~\ref{t.wembed}. The section concludes with relevant properties for general outer $L^p$ spaces which are needed in the paper.  Section~\ref{s.wellsep_lebesgue} discusses $L^2$ restriction estimates for the wave packet transform in upper 3-space. These estimates are  key to the proof of Theorem~\ref{t.wembed} which is pushed to Section~\ref{s.MainThmproof}.

\subsection{Notation}   
Given a finite interval  $I$ with center $c_I$, denote $aI$ as the interval with center $c_I$ and length $|aI| = a|I|$. For a fixed finite interval $I$, let
\begin{equation}\label{e.localization}
\chi_I(x) = \Big[1+\Big(\frac{|x-c_I|}{|I|}\Big)^{2}\Big]^{-1}.
\end{equation}
Let $\mathscr{S}(\R)$ denote the space of Schwartz functions on $\R$. We write the Fourier transform  of $f \in \mathscr{S}(\R)$ as
	$$ \widehat{f}(\xi) = \int_\R e^{-i \xi x} f(x) \, dx.$$ 
Given a weight $w: \R \to [0,\infty]$, let $w(E) = \int_E w(x) \, dx$  for all Lebesgue measurable sets $E$ on $\R$. When $E=(a,b)$ is an interval on $\R$, we write $w(a,b) = w((a,b))$ for convenience. We denote weighted $L^p$ spaces on $\R$ as $L^p(w)=L^p(\R,w)$ and the norm as $\|f \|_{L^p(w)} = \|f\|_{L^p(\R,w)}$ with similar convention for weak $L^p$ spaces. Finally, for a dyadic grid $\mathcal{D}$, we write the  dyadic (weighted) $L^p$ maximal functions over $\R$ as 
	\[ M_{p,w}(f)(x) = \sup_{\text{dyadic }Q \ni x} \left( \frac{1}{|Q|} \int_Q |f(x)|^p  \, w(x) d  x \right)^{1/p} \]
where $M = M_{1,1}$ is the standard dyadic  maximal function and $M_p = M_{p,1}$ is the dyadic  $L^p$ maximal function.

\section{Outer $L^p$ Spaces}\label{s.OuterLp}

This section sets up the outer $L^p$ space over upper 3-space in Theorem~\ref{t.wembed}. This setup is built upon the outer $L^p$  definitions and concepts formulated in  \cite{dt2015}. For convenience, we record useful properties of outer $L^p$ spaces in Section~\ref{subs.propertiesouterLp}.

\subsection{Outer $L^p$ Spaces over $X$} \label{subs.outerlp}
We work with the outer $L^p$ space associated with outer measure space $(X,\sigma,\s)$ where $X= \R \times \R \times \R_+$ is upper 3 space, $\sigma$ is a pre-measure on $X$ with respect to a distinguished collection of Borel sets $\E$, and $\s$ is a \textit{size}, i.e., a quasi sub-additive averaging map over each collection $E \in \E$. In keeping with the language developed in time-frequency analysis, the first coordinate of $X$ represents \textit{time}, the second coordinate represents \textit{frequency}, and the third coordinate represents \textit{scale}.

\subsubsection{Outer Measure Spaces and 3D Tents}
The distinguished collection of Borel sets in $X$ which our outer measure space is built over is the collection of \textit{3D tents} (or tents for short) in upper 3-space. Fix a triplet $\Theta = (C_1,C_2,b)$ such that $\min(C_1,C_2)> b>0$  where $b$ is a sufficiently small parameter to be used later. For each $(x, \xi, s)\in X$, define the 3D tent 
$$T_{\Theta}(x,\xi,s) := \left\{ (y,\eta,t) \in X \,\, : \,\, t < s , \, |y-x| < s-t , \, -\frac {C_1}t < \eta - \xi < \frac{C_2}t \right\} . $$
A tent $T_\Theta(x,\xi,s)$ is asymmetric in frequency unless $C_1=C_2$. An image of the center component in a generic 3D tent $T_\Theta(x,\xi,s)$ with $C_1\neq C_2$ is shown in Figure~\ref{f.tent1}. 

\begin{figure}
 \begin{center}
 \includegraphics[scale=0.27]{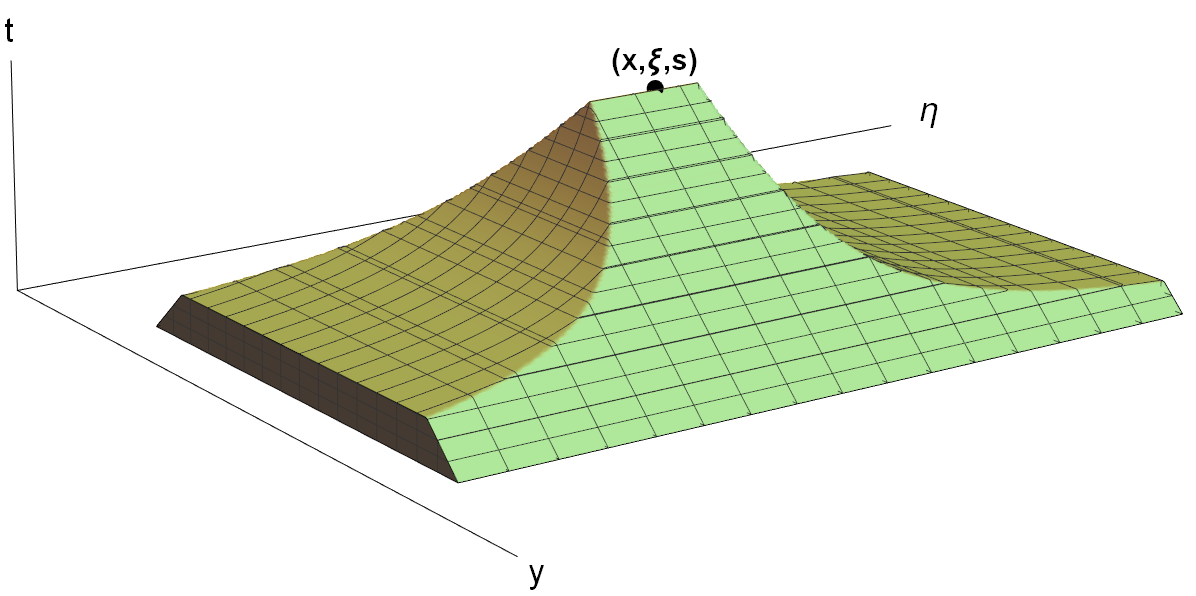}
 \caption{The center component of a generic 3D tent $T_\Theta(x,\xi,s)$ with $C_1 \neq C_2$. Note that the image is only a portion of the whole tent $T_\Theta(x,\xi,s)$ as tents have unbounded support in frequency. }
  \label{f.tent1}
 \end{center}
 \end{figure}
 
We further subdivide a 3D tent $T_\Theta(x,\xi,s)$ into \textit{core} and \textit{lacunary} components. The overlapping or core of the tent $T_\Theta(x,\xi,s)$,  denoted by $T^b_\Theta (x,\xi,s)$,  is defined as
$$T^b_\Theta (x, \xi, s):= \left\{(y,\eta,t) \in T_\Theta (x,\xi,s) \,\, : \,\,  |\eta-\xi|\le bt^{-1} \right\}.$$
The lacunary part of the tent $T_\Theta(x,\xi,s)$, denoted by $T^\ell_\Theta(x,\xi,s)$ is the asymmetric shell which is disjoint from the core,
$$T^\ell_\Theta (x,\xi,s):=T_\Theta (x,\xi,s)\setminus T^b_\Theta (x,\xi,s).$$ 
Figure~\ref{f.tent2} shows two-dimensional projections of a tent which helps distinguish the separation between the core and lacunary parts. The choice in $b< \min(C_1,C_2)$ is to ensure the shells of $T_\Theta$ are nontrivial. In the context of Theorem~\ref{t.wembed}, we set $\delta = 2^{-8} b$ for the frequency support of the kernel $\phi$ in the wave packet transform.

We now consider a pre-measure over the distinguished collection of 3D tents.  Fixing  $\Theta$, let $\E$ be the collection of all tents $T_\Theta$ in $X$. For a fixed weight function $w:\R\to [0,\infty]$, define the pre-measure $\sigma^w: \E \to [0,\infty)$ where 
$$ \sigma^w\big( T_\Theta(x,\xi,s)\big) := w(x-s,x+s) = \int_{\R} 1_{|u-x|<s} w(u) \, du . $$
To extend $\sigma^w$ to an outer measure $\mu^w$ on $X$, define for an arbitrary set $E \subset X$,
$$ \mu^w(E) := \inf \Big\{ \sum_{T_\Theta \in \E'} \sigma^w(T_\Theta) \,\, : \,\, E \subset \bigcup_{T_\Theta \in \E'} T_\Theta \Big\}$$
where the infimum is over all countable sub-collections $\E'$ of $\E$ which covers $E$. It is straightforward to check that $\mu^w(T_\Theta) = \sigma^w(T_\Theta)$ for all tents $T_\Theta \in \E$.

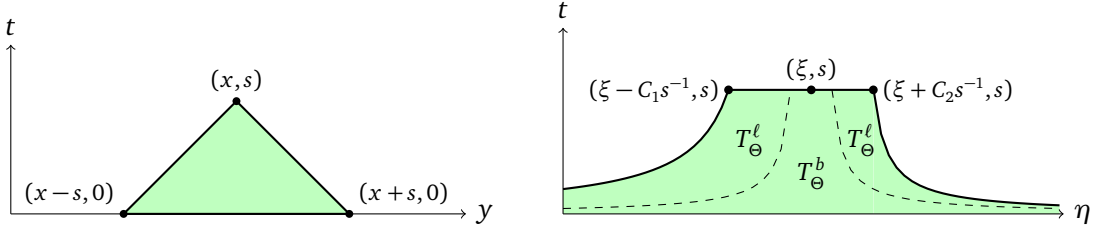
\begin{figure}
\begin{tikzpicture}[scale=0.5]
    \draw[->] (0,0) -- (12.1,0) node[right]{$y$};
    \draw[->] (0,0) -- (0,4.5) node[above]{$t$};
    \filldraw[thick,fill=green!25] (3,0.0) -- (6,3) -- (9,0.0) -- (3,0.0);
    \fill (6,3)  circle[radius=3pt] node[above]{\footnotesize{$(x,s)$}};
    \fill (9,0)  circle[radius=3pt] node[above right]{\footnotesize{$(x+s,0)$}};
    \fill (3,0)  circle[radius=3pt] node[above left ]{\footnotesize{$(x-s,0)$}};
\end{tikzpicture}
\hspace{1em}
\begin{tikzpicture}[scale=0.55]
    \fill[green!25] (4,3) -- (7.5,3)  -- (7.5,0) -- (4,0) -- (4,3);
    \fill [green!25, domain=7.5:12, variable=\x]
      (12, 0.207)
      -- (12, 0)
      -- (7.5, 0)
      -- (7.5, 3)
      -- plot (\x, {1/(\x-7.17)});
    \fill [green!25, domain=0:4, variable=\x]
      (0, 0.6)
      -- plot (\x, {3/(5-\x)})
      -- (4, 3)
      -- (4, 0)
      -- (0, 0);
    \draw [thick] (4,3) -- (7.5,3) ;
    \draw [thick, domain=7.5:12] plot (\x, {1/(\x-7.17)});
    \draw [thick, domain=0:4] plot (\x, {3/(5-\x)});
    \draw [dashed, domain=6.5:12] plot (\x, {0.75/(\x-6.25)}); 
    \draw [dashed, domain=0:5.5] plot (\x, {0.75/(5.75-\x)}); 
    \draw[->] (0,0) -- (12.1,0) node[right]{$\eta$};
    \draw[->] (0,0) -- (0,4.5) node[above]{$t$};
    \fill (6,3)  circle[radius=3pt] node[above]{\footnotesize{$(\xi,s)$}};
    \fill (7.5,3)  circle[radius=3pt] node[right]{\footnotesize{$(\xi + C_2s^{-1},s)$}};
    \fill (4,3)  circle[radius=3pt] node[left ]{\footnotesize{$(\xi - C_1s^{-1},s)$}};
    \draw (6,1)  node{\small{$T^b_\Theta$}};
    \draw (7.2,1.8)  node{\small{$T^\ell_\Theta$}};
    \draw (4.5,1.8)  node{\small{$T^\ell_\Theta$}};
\end{tikzpicture}
\caption{Projections of a tent $T_\Theta(x,\xi,s)$  with $C_1 \neq C_2$. The left image is the time-scale projection and the right image is the frequency-scale projection with the partition between the core $T^b_\Theta$ and lacunary $T^\ell_\Theta$ regions. }
\label{f.tent2}
\end{figure}

It remains to define a non-negative averaging operator called a \textit{size} on the space $\mathcal{B}(X)$ of  Borel-measurable functions over $X$. A size is a map $\s: \mathcal{B}(X) \to [0,\infty]^\E$ such that the following properties hold for all $F,G \in \mathcal{B}(X)$ and all $T_\Theta \in \E$.
    \begin{enumerate}
	\item \, [Monotone] If $|F| \le |G|$, then $\s(F)(T_\Theta) \le \s(G)(T_\Theta)$ .  
	\item \, [Scaling] If $\lambda \in \C$, then $\s (\lambda F)(T_\Theta) = |\lambda| \, \s(F)(T_\Theta).$ 
	\item \, [Quasi Triangle] There exists constant $C=C(\s) \ge 1$ such that
		\begin{equation}\label{e.size_quasitriangle}
		    \s(F + G)(T_\Theta) \le C \Big[ \s(F)(T_\Theta) + \s(G)(T_\Theta)\Big]  \,.
		\end{equation}
			The infimum of all such $C$ is the \textit{quasi-triangle constant of size $\s$}.
    \end{enumerate}
To construct the size in Theorem~\ref{t.wembed}, we first denote $S_{T_\Theta}$ as the continuous square function operator of a Borel function $F \in \mathcal{B}(X)$ restricted to the lacunary part of a fixed tent $T_\Theta$,
$$S_{T_\Theta}(F)(u)=\left(\int_{T^\ell_\Theta(x,\xi,s)} |F(y,\eta,t)|^2 1_{|y-u|<t} dyd\eta \frac{dt}{t} \right)^{1/2}. $$ 
The size $\s^w$ in Theorem~\ref{t.wembed} is then a superposition of an $L^\infty$ norm over the core of a tent and an $L^2(w)$ average norm for the square function $S_{T_\Theta}$. Formally,
$$\s^w(F)(T_\Theta(x,\xi,s)) :=\frac 1{\sqrt{w (x-s,x+s)}} \left\|S_{T_\Theta}(F) \right\|_{L^2(w)} +  \sup_{(y,\eta,t)\in T^b_\Theta(x,\xi,s)} |F(y,\eta,t)| $$
$$=\left(\frac 1{w(x-s,x+s)} \int_{T^\ell_\Theta(x,\xi,s)} |F(y,\eta,t)|^2  w(y-t,y+t)dyd\eta \frac{dt}{t}\right)^{1/2} + \sup_{(y,\eta,t)\in T^b_\Theta(x,\xi,s)} |F(y,\eta,t)|.$$

It is straightforward to check $\s^w$ is a size on $X$ with quasi-triangle constant $1$. The  triplet $(X,\sigma^w,\s^w)$  as defined above is therefore the outer measure space for this paper. Rather than denoting the space with $\mu^w$, we use the pre-measure $\sigma^w$ for it is implicitly in terms of the collection of tents $\E$. As we are working with a fixed $\Theta = (C_1,C_2,b)$,  we drop the $\Theta$ notation out of convenience and write a 3D tent as $T(x,\xi,s) = T_\Theta(x,\xi,s)$.  Be aware that the implicit constant  in Theorem~\ref{t.wembed} is also in terms of $\Theta$.

\subsubsection{Outer $L^p$ Spaces}
We  formulate the outer integrable spaces with respect to the outer measure space $(X,\sigma^w,\s^w)$. Given $\lambda>0$ and $F \in \mathcal{B}(X)$, define the \textit{super level measure} associated with $\mu^w$ by
$$\mu^w\big( \s^w(F) > \lambda\big) := \inf\Big\{ \mu^w(E) \, : \, E \subset X \text{ Borel \, \, s.t. \,  } \sup_{T \in \E} \s^w( F 1_{X \backslash E})(T) \le \lambda \Big\}. $$
For each  $0<p<\infty$ and $F \in \mathcal{B}(X)$, consider the \textit{outer $L^p$ maps}
\begin{align*}
    &\| F \|_{\mathcal{L}^{p}(X,\sigma^w,\s^w)} := \left( \int_0^\infty p \lambda^{p-1} \mu^w\big( \s^w(F) > \lambda\big) \, d\lambda \right)^{1/p}  \\[0.75em]
    &\|F\|_{\mathcal{L}^{p,\infty}(X,\sigma^w,\s^w)} := \sup_{\lambda > 0} \Big(\lambda^p \mu^w\big( \s^w(F)>\lambda \big)  \Big)^{1/p} \\[1em]
    & \|F\|_{\mathcal{L}^{\infty,\infty}(X,\sigma^w,\s^w)} = \|F\|_{\mathcal{L}^\infty(X,\sigma^w,\s^w)} := \sup_{T \in \E} \s^w(F)(T)
\end{align*} 
and set $\L^p(X,\sigma^w,\s^w)$, $\L^{p,\infty}(X,\sigma^w,\s^w)$ as the set of Borel functions whose corresponding outer $L^p$ map is finite. As in classical $L^p$ theory, $\L^{p}(X,\sigma^w,\s^w)$ is contained in $\L^{p,\infty}(X,\sigma^w,\s^w)$.

\subsection{Properties of outer $L^p$ spaces}\label{subs.propertiesouterLp}

We record  useful properties concerning outer $L^p$ spaces and their weak versions. The properties hold for general outer measure spaces so we use an abstract outer measure space $(X, \sigma ,\s)$ (where $X$ is a metric space) and outer $L^p$ space $\mathcal{L}^p(X,\sigma,\s)$.

The following proposition from \cite[Proposition 3.1]{dt2015} shows $\|\cdot\|_{\L^p(X,\sigma,\s)}$ is a quasi seminorm.
\begin{proposition} 
Let $(X,\sigma,\s)$ be an outer measure space. Consider  $F,G \in \mathcal{B}(X)$ and $0 < p \le \infty$. 
\begin{enumerate}
    \item If $|F| \le |G|$ then $\| F \|_{\mathcal{L}^p(X,\sigma,\s)} \le \| G \|_{\mathcal{L}^p(X,\sigma,\s)} $ 
    \item If $\lambda \in \C$, $\| \lambda F\|_{\mathcal{L}^p(X,\sigma,\s)} = \lambda \, \| F \|_{\mathcal{L}^p(X,\sigma,\s)}$ 
    \item Let $C$ be the quasi-triangle inequality of size $\s$. Then
    	\begin{equation}\label{e.quasitri}
    		 \| F + G \|_{\mathcal{L}^p(X,\sigma,\s)} \le  C_p  \Big( \, \| F \|_{\mathcal{L}^p(X,\sigma,\s)} + \| G \|_{\mathcal{L}^p(X,\sigma,\s)} \Big) 
	\end{equation}
    where $C_{p} = \begin{cases} 2^{1/p} C & 0 < p <  1 \\ 2 C & 1 \le  p < \infty \\ C & p = \infty  \end{cases}$ . 
\end{enumerate}
The above properties also hold for the weak space $\mathcal{L}^{p,\infty}(X,\sigma,\s)$. 
\end{proposition}

Recall size $\s^w$ has quasi-triangle constant of $1$. As such, both  $\mathcal{L}^{p}(X,\sigma^w,\s^w)$ and $\mathcal{L}^{p,\infty}(X,\sigma^w,\s^w)$  for $p>1$ have a quasi-triangle  constant of $2$.

Note the quasi-triangle inequality \eqref{e.quasitri}  can be generalized to a summation of $n$ functions $F_j$ by
    $$ \big\| \sum_{j=1}^n F_j \big\|_{\mathcal{L}^p(X,\sigma,\s)} \le \sum_{j =1}^n C_p^{j} \| F_j \|_{\mathcal{L}^p(X,\sigma,\s)}$$
with a similar inequality if the summation is with respect to  $\mathcal{L}^{p,\infty}(X,\sigma,\s)$. Assuming the sequence $\|F_j\|_{\L^p(X,\sigma,\s)}$ has sufficient decay when $j \to \infty$, this can be extended to an infinite series. One application is the following  domination property  presented by Uraltsev \cite[Corollary 2.1]{uraltsev2016}.

\begin{proposition}[Dominated Convergence]\label{p.outerDom}
Fix an outer measure space $(X,\sigma,\s)$ and $0 < p \le \infty$. Consider   Borel functions $F$, $F_j \in \mathcal{B}(X)$   satisfying the following properties.
    \begin{enumerate}
        \item $|F| \le \limsup_{j \to \infty} |F_j|$ pointwise on $X$.
        \item There exists $C_{p}' > C_p \ge 1$ where $C_p$ is a quasi-triangle constant for $\mathcal{L}^p(X,\sigma,\s)$  such that
            \begin{equation}\label{eq:DominatingConv_Cond2}
            \sup_{j\ge 1} (C_{p}')^{j} \,  \| F_{j+1}-F_j\|_{\mathcal{L}^p(X,\sigma,\s)} \lesssim \|F_1\|_{\mathcal{L}^p(X,\sigma,\s)}    .
            \end{equation}
    \end{enumerate}
Then $\|F\|_{\mathcal{L}^p(X,\sigma,\s)} \lesssim_{C_p,C_p'} \|F_1\|_{\mathcal{L}^p(X,\sigma,\s)}.$ Moreover, if $\|F_1 \|_{\mathcal{L}^p(X,\sigma,\s)} \lesssim C$ and the upper estimate in \eqref{eq:DominatingConv_Cond2} is replaced by $C$, then $\|F\|_{\L^p(X,\sigma,\s)} \lesssim C$.  A similar result holds in the context of weak outer $L^p$ spaces.
\end{proposition}

We finish by recording an outer $L^p$ version of classical Marcinkiewicz interpolation as shown in \cite[Proposition 3.5]{dt2015}.

\begin{proposition}[Outer Marcinkiewicz interpolation]\label{p.outerInterp} Let $(X,\sigma,\s)$ be an outer measure space and $(Y,\nu)$ be a measure space. Fix $1 \le p_1< p_2 \le \infty$. Let $F: L^{p_j}(Y,\nu) \to \mathcal{B}(X)$ be an operator for $j=1,2$ such that for any $f,g \in L^{p_1}(Y,\nu)+L^{p_2}(Y,\nu)$ and $\lambda\ge 0$ we have
	\begin{enumerate}
	\item $|F(\lambda f)| = |\lambda F(f)|$.
	\item $|F(f+g)| \le C (|F(f)|+|F(g)|)$ for some constant $C>0$.
	\item $\|F(f) \|_{\mathcal{L}^{p_j,\infty}(X,\sigma,\s)} \le A_j \| f \|_{L^{p_j}(Y,\nu)}$ for $j=1,2$.
	\end{enumerate}
Suppose $p \in (p_1,p_2)$. Then there exists a constant $C=C(p_1,p_2,p)>0$ such that 
	$$\| F(f) \|_{\mathcal{L}^p(X,\sigma,\s)} \le C A_1^{\theta_1} A_2^{\theta_2}  \| f \|_{L^p(Y,\nu)}$$ 
	where $0<\theta_1,\theta_2<1$ satisfy $\frac{1}{p} = \frac{\theta_1}{p_1}+\frac{\theta_2}{p_2}$. 
\end{proposition}

\section{$L^2$ restriction estimates for the wavelet projection operator}
\label{s.wellsep_lebesgue}

In this section we collect and prove some local $L^2$ estimates for the wavelet projection operator
$$P (f)(y,\eta,t)=\<f,\phi_{y,\eta,t}\>, \qquad (y,\eta,t)\in X:=\R\times \R \times \R_+,$$
and recall that $\phi_{y,\eta,t}$ is the $L^1$ normalized wave packet of $\phi$ at $(y,\eta,t)\in X$, defined in \eqref{e.L1wavepacket}. For the convenience of the reader, we recall this definition below,
$$\phi_{y,\eta,t}(x) = e^{-i\eta(y-x)} \frac 1 {t} \overline{\phi \Big(\frac{y-x}t \Big)},$$
here $\phi$ is often referred to as the mother wavelet in the literature.
Our methods actually work for more general setups, where the mother wavelet $\phi$ may depend on $(y,\eta,t)$, as long as it satisfies uniform decay estimates of Schwartz type and uniform frequency support conditions. For the simplicity of the presentation, the proof is only presented in our simpler setup where the wave packets all share the same mother wavelet $\phi$.

We look for geometric conditions on $Y\subset X$ such that there is a nontrivial improvement of the basic estimate
$$\Bigg( \int_Y |\<f,\phi_{y,\eta,t}\>|^2 \,  dy d\eta dt \Bigg)^{1/2} \lesssim \sup_{(y,\eta,t)\in Y} |\<f, \phi_{y,\eta,t}\>| \sqrt{|Y|}.$$
Here $|Y|$ is the 3D Lebesgue measure of $Y$. 
Note that if $Y$ is the lacunary region of a tent then   
$$\Bigg( \int_Y |\<f,\phi_{y,\eta,t}\>|^2 \,  dy d\eta dt \Bigg)^{1/2}  \lesssim \|f\|_2$$
thanks to Calder\'on Zymgund theory. Thus, we expect that the tent structure and $\|f\|_2$ will play an important role in the estimate and in the assumed geometric structure of $Y$. In fact, if we normalize $\|f\|_2=1$ then a first step is to obtain some geometric condition on $Y$ such that there is an improvement of the following nature
$$\Bigg( \int_Y |\<f,\phi_{y,\eta,t}\>|^2 dy d\eta dt \Bigg)^{1/2} \lesssim_\epsilon \Big(1+\sup_{(y,\eta,t)\in Y} |\<f, \phi_{y,\eta,t}\>| \sqrt{|Y|}\Big)^{1-\epsilon},$$
for some $\epsilon \in (0,1)$. This will be the main theme of this section.

\subsection{Discrete restriction estimates} 
We first consider a discretized variant of the above local $L^2$ norm, namely geometric conditions on a set of points $E\subset X$ such that there is some $\epsilon\in (0,1)$ such that
\begin{equation}\label{e.wellsep_discrete_target}
\Big(\sum_{(y,\eta,t)\in E} t|\<f,\phi_{y,\eta,t}\>|^2 \Big)^{1/2} \lesssim \|f\|_2^{\epsilon}\left(\|f\|_2+\sup_{(y,\eta,t)\in E} |\<f, \phi_{y,\eta,t}\>| \Big(\sum_{(y,\eta,t)\in E} t\Big)^{1/2}\right)^{1-\epsilon}.
\end{equation}

A preliminary result we need is the following standard estimate regarding inner products of wave packet; see \cite{thiele2006} for proof in phase plane analog.

\begin{lemma}\label{l.wavepacket1}
Let $\phi \in \mathscr{S}(\R)$. Consider points $(y,\eta,t)$, $(y',\eta',t') \in X$ with $t' \ge t$. Then for any integer $N\ge 1$, 
    $$ \left| \<\phi_{y,\eta,t},\phi_{y',\eta',t'} \> \right| \lesssim_{\phi,N}  (t')^{-1} \,  \Big[ 1+\Big(\frac{|y-y'|}{t'}\Big)^2 \Big]^{-N} . $$
\end{lemma}

We now define a notion of well-separation for a discrete collection of points in $X$ which is an extension of the analog of well-separation in phase plane analysis. 

\begin{definition}
A collection of points $E \subset X$ is {\bf well-separated} if there exists constants $\alpha,\beta > 0$ such that for all $(y,\eta,t),(y',\eta',t') \in E$, either 
    \begin{equation}\label{e.wellsep_pts}
        |y - y'| > \alpha \max(t,t') \qquad \text{ or } \qquad |\eta - \eta'| > \beta \max\big(t^{-1},(t')^{-1}\big) .
    \end{equation}
\end{definition}

Over a discrete collection of well-separated points in $X$, it is possible to obtain a version of estimate~\ref{e.wellsep_discrete_target} in the following form. 

\begin{lemma}\label{l.wellsep1} 
Fix $\phi \in \mathscr{S}(\R)$ such that $\supp \widehat{\phi} \subset (-\delta,\delta)$. Consider a countable collection  of well separated  points  $E = \{(y_k,\eta_k,t_k) \} \subset X$   with separation constants $\alpha>0$ and $\beta \ge 4 \delta$ as in \eqref{e.wellsep_pts}. Suppose $s \in (0,1)$. Then for any $f \in L^2(\R)$,
\begin{equation}\label{e.wellsep1}
    \Bigg(\sum_{k}  t_k |\<f,\phi_{y_k,\eta_k,t_k}\>|^2 \Bigg)^{1/2} \lesssim_{\phi,\alpha,s} \|f\|_2  + \Bigg[\sup_{k} \left|\<f,\phi_{y_k,\eta_k,t_k}\>\right| \Big(\sum_{k} t_k\Big)^{1/2} \Bigg]^{\, s} \|f\|_2^{1-s} .
\end{equation}     
\end{lemma}

The proof of Lemma~\ref{l.wellsep1} follows a format similar to corresponding proofs in the phase plane setting, cf.\ \cite{lt2000,osttw2012,dl2012sm}, and in the continuum setting \cite{dt2015}. We first prove the $s=1/3$ case by appealing to standard phase plane analysis on wave packets. The general $0<s<1$ case follows from a logarithmic argument.

\begin{proof}
Assume the sum of $t_k$ is finite else nothing to show. By scaling, assume without loss of generality that $\|f\|_{2}=1$. For each $k$, let  $\phi_k=\phi_{y_k,\eta_k,t_k}$ and denote $D = \sup_{k} |\< f, \phi_k\>|$. 

\textit{Case $s=1/3$}.   Let $A = \sum_k t_k |\<f,\phi_k\>|^2>0$. Apply H\"older's inequality to get
    	$$ A^2 \le \big\| \sum_k t_k \<f, \phi_k\> \, \phi_k \big\|_2^2 =    \sum_{k,j} t_k t_j  \, \<f,\phi_k\> \,\,  \<\phi_k,\phi_j\>  \,\, \<\phi_j,f\> \,. $$
Split the right-hand double summation into a  ``diagonal" term $\{8^{-1} t_k \le t_j \le 8 t_k\}$ and an ``off-diagonal" term $\{8 t_j < t_k\} \cup \{8 t_k < t_j\}$. By symmetry,  the summation is bounded by  
	\begin{align}
	\sum_{k,j \, : \, 8^{-1} t_k \le t_j \le 8 t_k}   t_k t_j  \, \<f,\phi_k\> \,\,  \<\phi_k,\phi_j\>  \,\, \<\phi_j,f\> \,  \label{e:Type1Diag}\\
		& \hspace{-3cm} + 2	\sum_{k,j  \, :8t_j < t_k}   t_k t_j  \, \<f,\phi_k\> \,\,  \<\phi_k,\phi_j\>  \,\, \<\phi_j,f\> \, .\label{e:Type1DiagOff}
	\end{align}
	
To estimate the diagonal summand \eqref{e:Type1Diag}, use symmetry to bound the smaller of the $f$-inner products with the large one to obtain
	\[ \sum_{k,j  \, : \, 8^{-1} t_k \le t_j \le 8 t_k}   t_k t_j | \, \<f,\phi_k\> \,\,  \<\phi_k,\phi_j\>  \,\, \<\phi_j,f\> \,|  \le 2 \sum_{k } t_k |\<f,\phi_k\>|^2 \Big(\sum_{j \, : \, 8^{-1} t_k \le t_j \le 8t_k} t_j|\< \phi_k , \phi_j \>|\Big).  \]
We now show the inner summation of wave packets is bounded uniformly over all $k,j$. For fixed $k$, assume without loss of generality that $\< \phi_k, \phi_j \> \neq 0$ for all $j$ terms. Application of Lemma~\ref{l.wavepacket1} and equivalence of scales $t_k \sim t_j$ yield the inner product estimate
    \[\sum_{j \, : \, 8^{-1} t_k \le t_j \le 8t_k} t_j|\< \phi_k , \phi_j \>| \lesssim_\phi  \sum_{j \, : \, 8^{-1} t_k \le t_j \le 8t_k} \Big[1+\Big|\frac{y_j-y_k}{t_k}\Big|^2\Big]^{-1}  \lesssim_\phi \sum_{m=0}^\infty (1+m^2)^{-1} \#E_k(m) \]
where
		\[ E_k(m) := \Big\{ (y,\eta,t) \in E  \, : \,  8^{-1} t_k \le t \le 8 t_k, \,\, \<\phi_k, \phi_{y,\eta,t}\> \neq 0,  \,\, \frac{m}{4} t_k \le |y_k-y| < \frac{m+1}{4} t_k \Big\}.\]
It remains to show  $\# E_k(m)$ is uniformly bounded over all $m$ and $k$.  The nonzero inner product condition between $\phi_k$ and any $\phi_{y,\eta,t}\in E_k(m)$ implies they have overlapping frequency support with  $|\eta_k - \eta| \le 9 \delta t_k^{-1}$. Thus, all $(y,\eta,t) \in E_k(m)$ must lie over a rectangular region with bounded area dependent on $\delta$. Moreover,  well separation of $E$ implies for any $(y,\eta,t)$, $(y',\eta',t') \in E_k(m)$ that either $|y - y'| >  2^{-3} \alpha t_k $ or $|\eta - \eta'| > 2^{-1} \delta {t_k}^{-1}$. Therefore, the order of $E_k(m)$ $\lesssim \alpha^{-1}$ is uniformly bounded for all $k,m$ and so the diagonal term \eqref{e:Type1Diag} has the estimate
	\begin{equation}\label{ineq:Type1DiagOff_1}
	\sum_{k,j  \, : \, 8^{-1} t_k \le t_j \le 8 t_k}   t_k t_j  \, \<f,\phi_k\> \,\,  \<\phi_k,\phi_j\>  \,\, \<\phi_j,f\> \, \lesssim_{\phi,\alpha} A.
	\end{equation}

Turing our attention to  the off-diagonal summand \eqref{e:Type1DiagOff},  application of the Cauchy Schwarz inequality yields the bound
	\[\sum_{k,j \,:\, 8 t_j < t_k}  t_k t_j  \, \<f,\phi_k\> \,\,  \<\phi_k,\phi_j\>  \,\, \<\phi_j,f\> \, \le  A^{1/2} H^{1/2} \]
where 
	\[ H = \sum_k \Big[ \sum_{j \, : \, 8 t_j < t_k} (t_k)^{1/2} t_j| \< \phi_k, \phi_j \> \,\, \<\phi_j,f\>|\Big]^2 \le D^2  \sum_k t_k \Big[ \sum_{j \, : \, 8 t_j < t_k} t_j| \< \phi_k, \phi_j \> |\Big]^2 . \]
Fix $k$ and without loss of generality assume $\< \phi_k , \phi_j \> \neq 0$ for all $j$ terms in inner summation above. By consideration of frequency support, $|\eta_k - \eta_j| < 2 \delta t_j^{-1} < \beta t_j^{-1}$ and so well separation dictates $|y_k - y_j| > \alpha t_k$. By  Lemma~\ref{l.wavepacket1},
$$ t_j \left| \<\phi_k,\phi_j \> \right| \lesssim_\phi  t_k^{-1} \,  \Big[ 1+\Big(\frac{y_k-y_j}{t_k}\Big)^2 \Big]^{-1} t_j \lesssim_\phi (t_k)^{-1} \int_{[y_j - \frac{\alpha}{4} t_j, y_j + \frac{\alpha}{4} t_j]}\Big[ 1+\Big(\frac{y_k-x}{t_k}\Big)^2 \Big]^{-1} \, dx.$$
We claim the intervals $[y_j - \frac{\alpha}{4} t_j, y_j + \frac{\alpha}{4} t_j]$ above  are all disjoint from each other. Consider another point $(y_\ell,\eta_\ell,k_\ell) \in E$ such that $8t_\ell < t_k$ and $ \<\phi_k , \phi_\ell\> \not=0$. As $\phi_k$ and $\phi_\ell$ have overlapping frequency support (just as $\phi_k$ and $\phi_j$ do),
    $$ |\eta_j - \eta_\ell| \le 2\delta (t_j^{-1} + t_\ell^{-1}) \le \beta \max(t_j^{-1},t_\ell^{-1}).$$
Well separation between $(y_j,\eta_j,t_j)$ and $(y_\ell,\eta_\ell,t_\ell)$ then requires  $|y_j - y_\ell| \ge \alpha \max(t_j,t_\ell)$ which prevents  the intervals from intersecting nontrivially. Hence all such intervals are disjoint and by summing over all such $j$,
$$\sum_k t_k \Big[ \sum_{j \, : \, 8 t_j < t_k} t_j| \< \phi_k, \phi_j \> |\Big]^2 \lesssim_\phi \sum_k t_k \left( \frac{1}{t_k} \int_\R \Big[ 1+\Big(\frac{y_k-x}{t_k}\Big)^2 \Big]^{-1} \, dx  \right)^2 \lesssim C_\phi \sum_k  t_k.$$

This gives us the  off-diagonal inequality
	\begin{equation}\label{ineq:Type1DiagOff_2}
	\sum_{k,j\, :8 t_j < t_k}   t_k t_j  \, \<f,\phi_k\> \,\,  \<\phi_k,\phi_j\>  \,\, \<\phi_j,f\> \,  \lesssim_{\phi} A^{1/2} D \, \big(\sum_{k} t_k \big)^{1/2}.
	\end{equation}
Applying the diagonal estimate \eqref{ineq:Type1DiagOff_1} and off diagonal estimate \eqref{ineq:Type1DiagOff_2} yields the bound
$$ A^2 \lesssim_\phi A +   D \, \big(\sum_{k} t_k \big)^{1/2} A^{1/2} $$
and the desired inequality for case $s=1/3$ follows by rearrangement.

\textit{Case $s \in (0,1)$.} Assume $D = \sup_{k} | \< f, \phi_k\>|$  is finite.   Subdivide the collection  $E = \bigcup_{j \ge 0} A_j$ where
	$$ A_j = \left\{k  \, : \, 2^{-(j+1)} D < |  \< f, \phi_k\>|  \le 2^{-j} D \right\} . $$
Let $A_{\ge j} = \bigcup_{i \ge j} A_i$ be the union of indices corresponding to the points in all levels below $A_j$. For fixed $j$, the sub-collection corresponding to $A_{\ge j}$ is still well separated. Apply the $s=1/3$ estimate to $A_{\ge j}$ to get
	$$ \big(\sum_{k \in A_{\ge j}} t_k |\<f,\phi_k\>|^2\big)^{1/2}  \lesssim_{\phi,\alpha} 1 +\Big[\big(  2^{-j} D\big) \big(\sum_{k  \in A_{\ge j}}   t_k \big)^{1/2}\Big]^{1/3} . $$
Note the right-hand side above decays to 1 for large $j$. By taking  $j \ge \max(0, \log_2[D (\sum_{k} t_k)^{1/2} ])$, 
	\begin{equation}\label{e.type1s_1}
	\sum_{k \in A_{\ge j}} t_k |\<f,\phi_k\>|^2 \lesssim_{\phi,\alpha} 1 .
	\end{equation}
 A similar bound exists over each individual level $A_j$. Indeed, the  points associated with each $A_j$ are themselves well separated so application of the $s=1/3$ case shows
	 $$ \big(\sum_{k \in A_j } t_k |\<f,\phi_k\>|^2\big)^{1/2} \lesssim_{\phi,\alpha} 1 +\Big[\big(  2^{-j} D\big) \big(\sum_{k  \in A_{j}}   t_k \big)^{1/2}\Big]^{1/3} \sim_{\phi,\alpha} 1 +\Big[ (\sum_{k \in A_j} t_k |\<f, \phi_k\>|^2)^{1/2}\Big]^{1/3}  $$
and by rearrangement,
	\begin{equation}\label{e.type1s_2}
	\sum_{k \in A_j} t_k |\<f,\phi_k\>|^2 \lesssim_{\phi,\alpha} 1 .
	\end{equation}
Fix the smallest positive integer $j \ge \max(0, \log_2[D (\sum_{k} t_k)^{1/2} ])$ and apply \eqref{e.type1s_1}-\eqref{e.type1s_2} to obtain the  logarithmic estimate 
	\begin{align}
	\begin{split}
	\sum_{k} t_k |\<f, \phi_k \>|^2  &=  \sum_{k \in A_{\ge j}}  t_k |\<f, \phi_k \>|^2   +  \sum_{i=0}^{j-1} \sum_{k \in A_i}  t_k |\<f, \phi_k \>|^2     \\
		&\lesssim_{\phi,\alpha} 1 +   \max\Big(0, \log_2[D (\sum_{k} t_k )^{1/2}]\Big). \label{e.wellsep1_log}    
	\end{split}
	\end{align}
The passage to arbitrary $s \in (0,1)$ uses the standard logarithmic properties  $s \log (x) = \log(x^s)$ and $\max(0,\log x) \le x$ for all $x > 0$, 	
    \begin{equation*}
     \big(\sum_k  t_k |\<f,\phi_k\>|^2\big)^{1/2} \lesssim_{\phi,\alpha,s} \left( 2s + \Big[D \big(\sum_k t_k \big)^{1/2}\Big]^{2s} \right)^{1/2} \lesssim_{\phi,s} 1 +  \Big[ D  \big(\sum_k t_k \big)^{1/2}\Big]^{\, s}. \qedhere
     \end{equation*}
\end{proof}

\subsection{Continuous restriction estimates}\label{subs.wellsep}

In this section, we describe a set of geometric conditions on a set $Y\subset X$ such that there is some nontrivial estimate for 
$$\int_Y |\<f,\phi_{y,\eta,t}\>|^2 \, dy d\eta dt.$$ 
As indicated at the beginning of this section, these conditions will involve some union of tents, or more precisely union of lacunary parts of tents. 

We now define a notion of well separation with regards to a collection of 3D  tents; assume  all tents  have the same parameterization $\Theta = (C_1,C_2,b)$. Below by \textit{partial tents} we mean subsets of tents.

\begin{definition}
Consider a collection of partial tents $E = \{ T^*_k\}$ where $ T^*_k \subset T_k = T(x_k,\xi_k,s_k)$ and the $T^*_k$ are pairwise disjoint from each other.  The collection $E$ is \textbf{well-separated} if there exists constants $\alpha \ge 1$, $\beta >0$, and $B > 1$ such that if $(y,\eta,t) \in T^*_k$ and $(y',\eta',t') \in \bigcup_j T^*_j$ satisfy $B t' < t$ then either 
    \begin{equation}\label{e.wellsep_partialtents}
        |\eta - \eta'| > \beta (t')^{-1} \qquad \text{ or } \qquad |x_k - y'| > \alpha (s_k - t') .
    \end{equation}
\end{definition}

The well separation   mandates points sufficiently small in scale with respect to a partial tent $T^*_k$ must either be sufficiently far in frequency from another point in $T^*_k$ else in another tent altogether.  
Note that when $T^*_k$ is a singleton set with coordinates comparable to the tent $T_k$ and $\alpha>1$, this reduces to well separation of points. It would be interesting to see an appropriate well separation definition for an arbitrary $Y \subset X$ but this is beyond the scope of the paper.

Similar to the well separation of points, the projection operator $P(f)$ over well separated partial tents is almost $L^2$.

\begin{lemma}\label{l.wellsep2}  Fix  $\phi \in \mathscr{S}(\R)$ such that $\supp \widehat{\phi} \subset (-\delta,\delta)$. Consider a collection of well separated partial tents $E = \{ T^*_k = T^*(x_k,\xi_k,s_k)\}$ with separation constants $\alpha \ge 1$, $\beta \ge 4\delta$, and $B>1$ as in \eqref{e.wellsep_partialtents}. Suppose $s \in (0,1)$. Then for any $f \in L^2(\R)$,
\begin{equation}
    \Bigg(\int_{ \bigcup_k T^*_k} |\<f, \phi_{y,\eta,t}\>|^2 dyd\eta dt \Bigg)^{1/2} \lesssim \|f\|_2  + \Bigg[\sup_{(y,\eta,t)\in  \bigcup_k T^*_k} \left|\<f, \phi_{y,\eta,t}\>\right| \Big(\sum_k s_k\Big)^{1/2}\Bigg]^{\, s} \|f\|_2^{1-s}
\end{equation}
where  the implicit constant depends on $\phi$, $\Theta$, $B$, $\beta$, and $s$. 
\end{lemma}  

The proof uses the same structure as in Lemma~\ref{l.wellsep1} where one proves the $s=1/3$ case and then generalizes to $0<s<1$. In order to show the general $0<s<1$ inequality, we need  the following simpler estimate.

\begin{lemma} \label{l.simplesquare} 

Fix $\phi \in \mathscr{S}(\R)$ such that $\supp \widehat{\phi} \subset (-\delta,\delta)$. Consider a subset $Y \subset \bigcup_k T_k^*$ with nonzero (three-dimensional) Lebesgue measure where the collection of partial tents $T^*_k$ is well-separated as in Lemma~\ref{l.wellsep2}.   
Then 
$$\Bigg(\int_Y |\<f,\phi_{y,\eta,t}\>|^2 dy d\eta dt\Bigg)^{1/2} \lesssim \|f\|_2 + \Bigg[\sup_{(y,\eta,t)\in Y} \left|\<f, \phi_{y,\eta,t}\>\right|  \,  \sqrt{|Y|} \Bigg]^{1/3}\|f\|_2^{2/3}.$$

\end{lemma}

\begin{proof}[Proof of Lemma~\ref{l.simplesquare}] By scaling invariant we may assume that $\|f\|_2=1$. Let $A=\int_Y  |\<f,\phi_{y,\eta,t}\>|^2 dy d\eta dt$. By Cauchy Schwarz,
$$A^2 \lesssim \Big\| \int_Y \<f,\phi_{y,\eta,t}\>  \, \phi_{y,\eta,t}dyd\eta dt \Big\|_2^2$$
$$=\int_{Y \times Y}\ \< f,\phi_{y,\eta,t} \>  \,\, \<\phi_{y,\eta,t},\phi_{y',\eta',t'}\>  \,\, \<\phi_{y',\eta',t'},f\> \, dy \,d\eta \, dt \, dy' \, d\eta' \, dt'.$$
Split the integral into three parts: the diagonal $ B^{-1} t'\le t\le Bt'$, the upper  half $t> Bt'$, and the lower half $t'>Bt$. We may estimate  the diagonal part by
$$\le 2  \int_Y |\<f,\phi_{y,\eta,t}\>|^2 \left[\int_{(y',\eta')\in \R^2, \, B^{-1}t'\le t\le Bt'} |\<\phi_{y,\eta,t},\phi_{y',\eta',t'}\>| dy'd\eta' dt' \right] dyd\eta dt .$$

\begin{claim}\label{cl.diagonal} We have the following uniform estimate for all $(y,\eta,t) \in Y$. 
$$ \int_{ B^{-1}t \le t' \le Bt} \int_{\R^2} |\<\phi_{y,\eta,t} ,  \phi_{y',\eta',t'}\>| \, dy' d\eta' \, dt' \lesssim 1$$
\end{claim} 
Using Claim~\ref{cl.diagonal}, it is clear that the diagonal part is $O(A)$. To see Claim~\ref{cl.diagonal} is true, we first note that by taking supremum of $t'$ over $B^{-1}t \le t' \le Bt$ and integrating $dt'$ over that region, we obtain 
 $$\int_{ B^{-1}t \le t' \le Bt} \int_{\R^2} |\<\phi_{y,\eta,t} ,  \phi_{y',\eta',t'}\>| \, dy' d\eta' \, dt'  \lesssim_{B} \sup_{ B^{-1}t \le t' \le Bt} \int_{\R^2} t \, |\< \phi_{y,\eta,t} , \phi_{y',\eta',t'} \>| \, dy' d\eta'  $$
To integrate frequency $\eta'$, recall that nonzero inner poduct of wave packets $\phi_{y,\eta,t}$ and $\phi_{y',\eta',t'}$  implies they have overlapping support in frequency. Restricing to such wave packets, it follows that  $\eta'$ has to be contained in an interval of length comparable to  $t^{-1}$. Thus,  the last display is bounded from above by
	$$ \lesssim_{B,\beta}   \sup_{ B^{-1}t \le t' \le Bt} \,  \sup_{\eta' \in \R} \Big(   \int_{\R}   |\<\phi_{y,\eta,t}, \phi_{y',\eta',t'}\>| \, dy' \Big) $$
	$$ \lesssim_{\phi,B,\beta}   \sup_{ B^{-1}t \le t' \le Bt} \,  \sup_{\eta' \in \R} \Big( t^{-1}  \int_{\R} \Big[1+\big|\frac{y-y'}{t}\big|^2\Big]^{-1} \, dy' \Big) \,\, \lesssim_{\phi,B,\beta} 1 $$
where Lemma~\ref{l.wavepacket1} was used in the passage to the second line. This establishes Claim~\ref{cl.diagonal} and the estimate on the diagonal terms.

For the off-diagonal terms, we show the details for the upper half term where $t>Bt'$ noting the lower half term may be treated similarly. Denote $D=\sup_{(y,\eta,t)\in Y}|\<f,\phi_{y,\eta,t}\>|$ for notation convenience.  By Cauchy-Schwarz,
$$\int_{Y} \<f,\phi_{y,\eta,t}\> \int_{(y',\eta',t')\in Y: B t'<t}\<\phi_{y,\eta,t},\phi_{y',\eta',t'}\> \,\, \<\phi_{y',\eta',t'},f\> \, dy'd\eta' dt' dyd\eta dt   $$
$$\le A^{1/2} \, \left(\int_Y \Big(\int_{(y',\eta',t')\in Y: B t'<t}  \left|\<\phi_{y,\eta,t},\phi_{y',\eta',t'}\> \,\, \<\phi_{y',\eta',t'},f\>\right| dy'd\eta' dt' \Big)^2 dyd\eta dt \right)^{1/2}$$
$$\le A^{1/2} D \left( \int_Y \Big(\int_{(y',\eta',t')\in Y: B t'<t} |\<\phi_{y,\eta,t},\phi_{y',\eta',t'}\>| dy'd\eta' dt' \Big)^2 dyd\eta dt \right)^{1/2}.$$
It remains to show 
$$ \sup_{(y,\eta,t) \in Y} \int_{(y',\eta',t')\in Y: B t'<t} |\<\phi_{y,\eta,t},\phi_{y',\eta',t'}\>| dy'd\eta' dt' \lesssim_\phi 1 .$$
Fix $(y,\eta,t) \in Y$. As $Y\subset \bigcup_j T^*_j$, it suffices to show the following claim.
\begin{claim} \label{cl.off-diagonal} We have the following uniform estimate for all $k$ and $(y,\eta,t) \in T^*_k$,
$$\sum_{j} \int_{T^*_j \,:\, B t' \le t} |\< \phi_{y,\eta,t} , \phi_{y',\eta',t'}\>| \, dy' d\eta' dt'  \lesssim_{\phi}   \Big[ 1+\Big(\frac{s_k-|x_k-y|}{t}\Big) \Big]^{-1}.$$
\end{claim}

Fix $(y,\eta,t) \in T^*_k$ and consider $(y',\eta',t') \in T^*_j$ such that $B t' \le t$. We may assume without loss of generality that $\< \phi_{y,\eta,t} , \phi_{y',\eta',t'} \> \neq 0$ which means their Fourier transforms have overlapping support. As $B t' < t$, 
    $$ |\eta - \eta'| \le 2 \delta (t')^{-1} \le \beta (t')^{-1}$$
so we are integrating $d\eta'$ with respect to an interval of length $2\beta (t')^{-1}$. 
By the well-separation criteria, we  can also restrict  the integral $d y'$ to the region $|y' - x_k| > s_k - t$.  

It remains to consider the integral with respect to $d t'$. For fixed position $y'$, we claim that any points of the form $(y',\eta'',t'') \in \bigcup_j T^*_j$ with $Bt''<t$ and $\<\phi_{y,\eta,t},\phi_{y',\eta'',t''}\>\neq 0$ must be within a factor of $B$ of each other in scale. Indeed, recall $(y',\eta',t') \in T^*_j$ and consider $(y'',\eta'',t'') \in T^*_m$ where $\< \phi_{y,\eta,t} , \phi_{y'',\eta'',t''} \> \neq 0$  and $B t'' \le t$. If we were to assume  $B t'' \le t'$ then it follows that 
    $$|\eta' - \eta''| \le  4 \delta (t'')^{-1} \le \beta (t')^{-1}$$
and so well-separation dictates 
	$$|y'' - x_j| > s_j - t'' > s_j - t' \ge |y' - x_j|.$$
This shows $y'' \neq y'$ so  setting $y'' = y'$ implies that $t'' \sim_B t'$.

Given $y' \in \R$, let $E(y')$ be the collection of scales $t''$ for which there exists $\eta''$ and $T^*_j$ such that $(y',\eta'',t'') \in T^*_j$ with $B t'' < t$ and $\< \phi_{y,\eta,t}, \phi_{y',\eta'',t''} \> \neq 0$. For each $y' \in \R$ with $E(y')\neq \emptyset$, there exists an interval $I(y') = [\alpha(y'), B \alpha(y')]$ which contains $E(y')$. 
Here, one can check
	    $$\alpha(y) := B^{-1} \sup_{t'' \in E(y')} t'' < \infty.$$
	
Applying  Lemma~\ref{l.wavepacket1} the well separation observations above with yields
	$$  \sum_{j} \int_{T^*_j \,:\, B t' \le t} |\< \phi_{y,\eta,t} , \phi_{y',\eta',t'}\>| \, dy' d\eta' dt'  $$ \\[-2.6em]
	\setlength{\jot}{12pt}
	\begin{align*}
	    &\lesssim \int_{|y' - x_k| > s_k -t} \int_{I(y')}  \int_{[\eta - \beta/t',\eta+\beta/t']}  |\< \phi_{y,\eta,t} , \phi_{y',\eta',t'} \>| \,  d\eta' dt' dy' \\
	    &\lesssim_B    \int_{|y' - x_k| > s_k -t} \Bigg(\sup_{t' \in I(y')}\int_{[\eta - \beta/t',\eta+\beta/t']}  t' |\< \phi_{y,\eta,t} , \phi_{y',\eta',t'} \>| \,  d\eta' \Bigg)  dy' \\
	    &\lesssim_{\phi,B,\beta} \int_{|y' - x_k| > s_k -t} t^{-1} \Big[ 1+\Big(\frac{|y-y'|}{t}\Big)^2 \Big]^{-2} dy' \lesssim  \Big[ 1+\Big(\frac{s_k-|x_k-y|}{t}\Big) \Big]^{-1}
	\end{align*}
as desired, verifying Claim~\ref{cl.off-diagonal}. Collecting the diagonal and off-diagonal estimates, we have
$$A^2 \lesssim A + A^{1/2} D |Y|^{1/2}$$
and from here the desired result follows. 
\end{proof}

\begin{proof}[Proof of Lemma~\ref{l.wellsep2} ]
Assume  $D = \sup_{(y,\eta,t)\in \cup_k T^*_k} |\< f , \phi_{y,\eta,t} \>|$ and the sum of $s_k$ are finite. By scaling, we may assume without loss of generality that $\|f\|_{2}=1$.

\textit{Case $s=1/3$}.   Let $A :=  \sum_{k} \int_{T^*_k} |\<f , \phi_{y,\eta,t} \>|^2 \, dy d\eta dt$. By Cauchy-Schwarz,
	$$A^2  \lesssim \sum_{k,j} \int_{T^*_k \times T^*_j}\<f , \phi_{y,\eta,t} \> \,\, \<\phi_{y,\eta,t}, \phi_{y',\eta',t'} \> \,\, \<\phi_{y',\eta',t'}, f \> \, dy d\eta dt \, dy' d\eta' dt'.$$
Split the integral  the diagonal part $ B^{-1} t'\le t\le Bt'$, the upper  half $t> Bt'$, and the lower half $t'>Bt$. It remains to show the diagonal part  is bounded above by $O(A)$ and the off diagonal parts are bounded above by 
$$O\left(\Big[\sup_{(y,\eta,t)\in \cup_k T^*_k} |\< f , \phi_{y,\eta,t} \>| \Big] \, \Big( \sum_k s_k\Big)^{1/2}  A^{1/2}\right).$$
This would imply
    $$ A^2 \lesssim_{\phi} A + \Big[\sup_{(y,\eta,t)\in \cup_k T^*_k} |\< f , \phi_{y,\eta,t} \>| \Big] \, \Big( \sum_k s_k\Big)^{1/2}  A^{1/2}$$
and the $s=1/3$ case follows by rearrangement. 

To estimate the diagonal part, we have by symmetry
	\begin{align*}
	&  2  \sum_{k,j=1} \int_{T^*_k \times T^*_j \, : \, B^{-1} t \le t' \le Bt} |\<f , \phi_{y,\eta,t} \>|^2 \, |\<\phi_{y,\eta,t} , \phi_{y',\eta',t'} \>| \, dy d\eta dt\, dy' d\eta' dt' \\
	&\le 2A \sup_{k, (y,\eta,t) \in T^*_k} \left( \sum_{j} \int_{T^*_j \, : \, B^{-1}t \le t' \le Bt} |\<\phi_{y,\eta,t} , \phi_{y',\eta',t'} \>| \, dy' d\eta' \, dt'  \right) \\
	&\le 2A \sup_{k, (y,\eta,t) \in T^*_k} \left(   \int_{ B^{-1}t \le t' \le Bt} \int_{\R^2} |\<\phi_{y,\eta,t} ,  \phi_{y',\eta',t'}\>| \, dy' d\eta' \, dt' \right) 
	\end{align*}
where last part is due the disjointness of $T^*_j$. Using Claim~\ref{cl.diagonal}, the last display is clearly $O(A)$.  

We now focus on the off-diagonal terms. We will prove the desired estimate for the upper half and the lower half will follow by similar construction. Taking note that the partial tents $T^*_k$ are pairwise disjoint, apply Cauchy-Schwarz to obtain estimates
$$ \sum_{k,j} \int_{T^*_k \times T^*_j} \< f , \phi_{y,\eta,t} \>  \,\, \< \phi_{y,\eta,t} , \phi_{y',\eta',t'} \> \,\,  \< \phi_{y',\eta',t'} , f \> \, dy d\eta dt \, dy' d\eta' dt' \le A^{1/2}  \Big(\sum_k H_k \Big)^{1/2} $$
where
	\begin{align*}
	H_k &= \int_{T^*_k} \Bigg( \sum_{j} \int_{T^*_j \,:\, B t' \le t} |\< \phi_{y,\eta,t} , \phi_{y',\eta',t'}\> \,\, \<\phi_{y',\eta',t'} , f\>| \, dy' d\eta' dt' \Bigg)^2 \, dy d\eta dt \\
		&\le \bigg[\sup_{(y,\eta,t)\in \cup_k T^*_k} |\< f , \phi_{y,\eta,t} \>| \bigg]^2 \int_{T^*_k} \Bigg( \sum_{j} \int_{T^*_j \,:\, B t' \le t} |\< \phi_{y,\eta,t} , \phi_{y',\eta',t'}\>| \, dy' d\eta' dt' \Bigg)^2 \, dy d\eta dt
	\end{align*}
For simplicity in notation, let $D$ be the supremum above.  Using Claim~\ref{cl.off-diagonal}  above, we have 
	$$ H_k   \lesssim_{\phi,B,\beta}  D^2 \int_{T^*_k} \Big[ 1+\Big(\frac{s_k-|x_k-y|}{t}\Big) \Big]^{-2}  \, dy d\eta dt  $$
	$$ \lesssim_{\phi,B,\beta} D^2 \int_0^{s_k} \int_{x_k-s_k}^{x_k+s_k} \int_{\xi - C_1 t^{-1}}^{\xi+C_2 t^{-1}}\Big[ 1+\Big(\frac{s_k-|x_k-y|}{t}\Big) \Big]^{-2}  \, d\eta dy dt   \lesssim_{\phi,\Theta} D^2 s_k $$ 
Summing over all $k$ gives 
    $$\Big(\sum_k H_k\Big)^{1/2} \lesssim_{\phi} \Big[\sup_{(y,\eta,t)\in \cup_k T^*_k} |\< f , \phi_{y,\eta,t} \>| \Big] \, \Big( \sum_k s_k\Big)^{1/2} $$
and the desired off-diagonal estimate now follows. 

\textit{Case $s \in (0,1)$}.  Generalization to $s \in (0,1)$ is essentially the same argument as in Lemma~\ref{l.wellsep1} with appropriate  modifications. Recall $D = \sup_{(y,\eta,t)\in \cup_k T^*_k} |\< f , \phi_{y,\eta,t} \>| < \infty$. Subdivide the collection $E$ of partial tents into 
	$$ A_j = \Big\{ (y,\eta,t)\in \bigcup_{k} T^*_k \, : \, 2^{-(j+1)} D < |\< f, \phi_{y,\eta,t}\>|  \le 2^{-j} D \Big\} . $$
Denote $A_{\ge j} = \bigcup_{i \ge j} A_i$, and for convenience let   $A_{\ge j,k}=T^*_k\cap A_{\ge j}$. For fixed $j$, the sub-collection $\{A_{\ge j,k}\}$ of partial tents is  well-separated so by applying the $s=1/3$ estimate,
	$$ \Bigg(\sum_k \int_{A_{\ge j,k}} |\<f,\phi_{y,\eta,t}\>|^2 \, dy d\eta dt\Bigg)^{1/2} \lesssim_{\phi,B,\beta} 1 + \Bigg[\Big(  2^{-j} D\Big) \Big(\sum_{k}   s_k  \Big)^{1/2}\Bigg]^{1/3} . $$
Note the upper estimate above decays such that for  $j \ge \max(0, \log_2[D (\sum_{k} s_k)^{1/2} ])$,
	\begin{equation}\label{e.type2s_1}
	 \int_{A_{\ge j}} |\<f,\phi_{y,\eta,t}\>|^2 \, dyd\eta dt \lesssim 1 .
	 \end{equation}

To estimate each  individual level $A_j$, use Lemma~\ref{l.simplesquare} to obtain
	 $$2^{-(j+1)}D |A_j|^{1/2}\le \Big(\int_{A_j} |\<f,\phi_{y,\eta,t}\>|^2 \, dy d\eta dt\Big)^{1/2}  \lesssim_{\phi,B,\beta} 1 +\left[\Big(  2^{-j} D\Big) |A_j|^{1/2}\right]^{1/3} . $$
From this estimate, it follows immediately that $2^{-j}D|A_j|^{1/2} \lesssim 1$, and consequently,
	\begin{equation}\label{e.type2s_2}
	\Big(\int_{A_j} |\<f,\phi_{y,\eta,t}\>|^2 \, dy d\eta dt\Big)^{1/2}  \lesssim 1.
	\end{equation}

Fix the smallest integer $j\ge \max(0, \log_2[D (\sum_{k} s_k)^{1/2} ])$. By  \eqref{e.type2s_1} and \eqref{e.type2s_2}, we obtain the following logarithmic estimate spanning all of $ \bigcup_k T_k^* =  \bigcup A_j$.
 	\begin{align*}
 	 \sum_k \int_{T_k^*} |\<f,\phi_{y,\eta,t}\>|^2 \, dy d\eta dt &= \int_{A_{\ge j}} |\<f,\phi_{y,\eta,t}\>|^2 \, dy d\eta dt + \sum_{k=1}^{j-1} \int_{A_k} |\<f,\phi_{y,\eta,t}\>|^2 \, dy d\eta dt \\
 	 	&\lesssim_{\phi,B,\beta} 1 +   \max\Big(0, \log_2[D (\sum_{k} s_k )^{1/2}]\Big).
 	 \end{align*}
The desired result for $s \in (0,1)$ is then obtained by 
using the same logarithmic argument as in the conclusion of Lemma~\ref{l.wellsep1}.
\end{proof}

\section{Proof of Theorem~\ref{t.wembed}}\label{s.MainThmproof}

We start with reductions which reduces the proof of Theorem~\ref{t.wembed} and make some useful observations. The outline of the proof will be then stated in Section~\ref{subs.proofidea} followed by the specifics of the proofs in Sections~\ref{subs.weakLinfty}-\ref{subs.weakLq}.

\subsection{Preliminary reductions and observations }\label{subs.reduct}

\subsubsection{Reduction - Discrete Parameter Tents}

Following \cite{dt2015}, we pass our tent paramterizations to a discrete subset of $X$ and prove Theorem~\ref{t.wembed} in the discrete parameter setting. The  chosen discrete subset is chosen such that tents in $\E$ are \textit{centrally contained} within the discrete parameter tents. Note that a point $(x',\xi',s') \in T(x,\xi,s)$ is centrally contained if  the following inequalities hold:
\begin{equation}\label{e.centralcontain}
	2^{-3}s \le s' \le 2^{-2}s \, , \qquad  |x'-x|\le 2^{-4}s \, , \qquad  |\xi'-\xi|\le 2^{-8} bs^{-1}.
\end{equation}

Consider the subset $X_{\Delta}$  of points $(x,\xi,s)\in X$ such that there exists integers $k,m,n$ satisfying 
$$x=2^{k-4}n, \ \ \xi=2^{-k-8}bm, \ \ s=2^k.$$
Let $\E_{\Delta}$ be the collection of all tents $T(x,\xi,s)$ with $(x,\xi,s)\in X_{\Delta}$.
We recall the following relation \cite[Lemma 5.2]{dt2015} regarding tents in $\E$ and $\E_\Delta$.
\begin{lemma}\label{l.centralcontain} 
For any $(x',\xi',s')\in X$ there exists $(x,\xi_-, s)$, $(x,\xi_+,s) \in X_{\Delta}$, such that $(x',\xi',s')$ is centrally contained in the tents $T(x,\xi_-s)$, $T(x,\xi_+,s)$ as specified in \eqref{e.centralcontain} and satisfy 
$$T(x',\xi',s') \subset T(x,\xi_-,s) \cup T(x,\xi_+,s),$$
$$T(x',\xi',s') \cap T^b(x,\xi_-,s)\cap T^b(x,\xi_+,s) \subset T^b(x', \xi',s').$$
\end{lemma}

Let $\s^w_\Delta$, $\sigma^w_\Delta$, and  $\mu^w_\Delta$ be  $S^w$, $\sigma^w$, and $\mu^w$ respectively restricted to the generating sub-collection $\E_\Delta$.  Given $T \in \E$, there exist $T^+,T^- \in \E_\Delta$ satisfying $T \subset T^+ \cup T^-$ and $$\sigma^w(T) \le \sigma^w_{\Delta}(T^+)+\sigma^w_{\Delta}(T^-) \le C \sigma^w(T)$$
where right-most inequality uses the doubling property of $w$. This implies the outer measures $\mu^w$ and $\mu^w_{\Delta}$ are equivalent. Furthermore, given $F \in \mathcal{B}(X)$ and tent $T \in \E$, application of $T^+,T^- \in \E_\Delta$ as above shows $$ \s^w(F)(T) \le C \Big( \s^w_\Delta(F)(T^+) + \s^w_\Delta(F)(T^-)\Big)$$
where $C$ is dependent on the doubling constant of $w$.

We therefore have an equivalence of the outer $L^p$ spaces and weak outer $L^p$ spaces $$\mathcal{L}^p(X,\sigma^w,\s^w) \sim \mathcal{L}^p(X,\sigma^w_\Delta,\s^w_\Delta), \qquad \mathcal{L}^{p,\infty}(X,\sigma^w,\s^w) \sim \mathcal{L}^{p,\infty}(X,\sigma^w_\Delta,\s^w_\Delta)$$ for $1 \le p \le \infty$. Thus, to prove Theorem~\ref{t.wembed}, it suffices to establish the following theorem with respect to the discrete parameter setting.  

\begin{theorem}\label{t.wembed-d}
Let $\phi \in \mathscr{S}(\R)$ such that $\supp \widehat{\phi} \subset (-2^{-8}b,2^{-8}b)$. Given a locally integrable function $f$ on $\R$, let $P(f)$ be the wave packet transform \eqref{e.3DPoisson}.  Suppose $2<q<\infty$ and $w \in \A_{q/2}$. Then
$$\|P(f)\|_{\mathcal L^{q}(X, \sigma^w_\Delta , \s^w_{\Delta})} \lesssim_{\Theta,\phi,q,[w]_{\A_{q/2}}} \|f\|_{L^{q}(w)}.$$
\end{theorem}

\subsubsection{Useful Observations}

\begin{remark}[Tent Containment]\label{r.obs_enlarge}
The proof of Theorem~\ref{t.wembed-d} requires us to pass from a collection of selected tents to another while maintaining well separation in the style of Section~\ref{s.wellsep_lebesgue}. To that end, we use the following geometric observation regarding tents. 

Given $x' \in \R$ and $s'>0$, denote the triangular strip
$$E_{x',s'} = \{ (y,\eta,t) \in X \, : \, 0<t<s', |y-x'|<s'-t\} $$
Consider a tent $T(x,\xi,s) \in \E$ such that $T(x,\xi,s) \cap E_{x',s'}$ is nonempty. Then the intersection itself is a 3D tent $T(x'',\xi,s'') \in \E$ where $(x''-s'',x''+s'') = (x-s,x+s) \cap (x'-s',x'+s')$.  
Figure~\ref{f.tentcontain} illustrates this intersection. As the frequency parameters play no role in this observation, the containment extends to the core and lacunary partial tents:
$$ T^b(x,\xi,s) \cap E_{x',s'} = T^b(x'',\xi,s'') \qquad \text{ and } \qquad T^\ell(x,\xi,s) \cap E_{x',s'} = T^\ell(x'',\xi,s'') .$$
 
\end{remark}

\begin{figure}
 \includegraphics[scale=0.35]{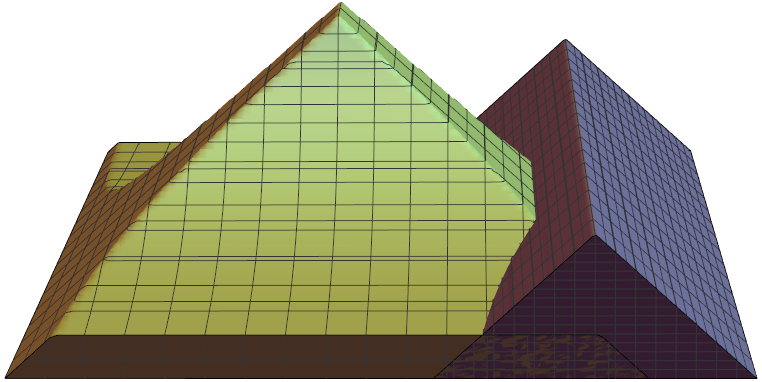}
 \hspace{1em}
 \includegraphics[scale=0.35]{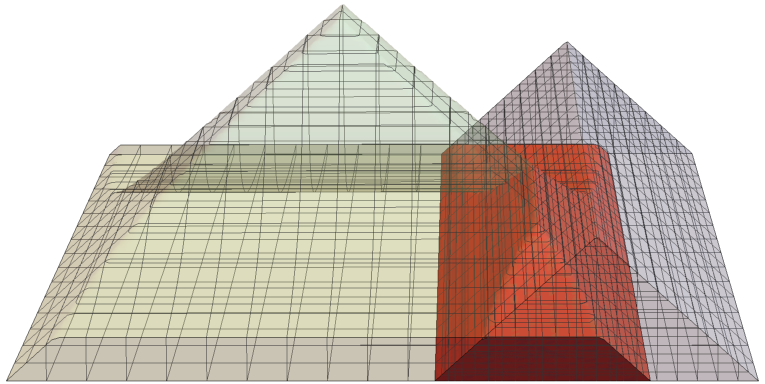}
\caption{The intersection of a 3D tent (in green) and a triangular strip (in blue) is another 3D tent (in red).}
\label{f.tentcontain}
\end{figure}

\begin{remark}[$\A_p$ weights are not $L^1$ integrable]\label{r.ApNotL1}
We briefly point out the standard fact that $\A_p$ weights are not integrable on $\R^n$.  Otherwise, reverse H\"older's inequality implies for all cubes $Q$
    $$ \int_Q \big[ w(x) \big]^r \, dx \lesssim_r  |Q|^{1-r} $$
where $r>1$ is dependent on $w$. Application of monotone convergence theorem shows $w=0$ a.e.\ which isn't an $\A_p$ weight. We state this remark as we haven't found another reference mentioning it.
\end{remark}

\begin{remark}[Shifted Dyadic Intervals]\label{r.3grids}
While the discrete parameter tents are not necessarily over dyadic intervals, we will implement another pass to work with dyadic grids on $\R$. The dyadic intervals we are interested are
    $$ \mathcal{D}_k = \left\{ \Big[2^{-j}(m+ (-1)^jk/3) \, , \, 2^{-j}(m+1+(-1)^jk/3)   \Big) \, : \, j,m \in \Z \right\} \qquad k \in \{0,1,2\} $$
where $\mathcal{D}_0$ is the standard dyadic grid, $\mathcal{D}_{1}$ is the 1/3 shifted grid, and  $\mathcal{D}_{2}$ is the 2/3 shifted grid.   We recall the three grids lemma (see \cite{christ1988}) where a finite interval $I \subset \R$ is comparable to a dyadic interval $J \in  \mathcal{D}_j$ for some $j \in \{0,1,2\}$ such that $I \subset J$ and $3 |I| \le |J| \le 6 |I|$. 
\end{remark}


\subsection{Structure of Proof}\label{subs.proofidea}
Fix $2<q<\infty$ and weight $w \in \A_{q/2}$. To prove Theorem~\ref{t.wembed-d}, it suffices to establish  the following weak outer $L^p$ estimates.
    \begin{align}
    &\| P (f)\|_{\mathcal{L}^{\infty}(X,\sigma^w_\Delta,\s^w_\Delta)} \lesssim_{\phi,q,[w]_{\A_{q/2}}} \| f\|_{L^\infty(w)}  \label{e.weakwembed-infty}  \\
    &\| P (f)\|_{\mathcal{L}^{q,\infty}(X,\sigma^w_\Delta,\s^w_\Delta)} \lesssim_{\phi,q,[w]_{\A_{q/2}}} \| f\|_{L^q(w)}  \label{e.weakwembed-q}      
    \end{align} 
If true, reverse H\"older's inequality  says \eqref{e.weakwembed-q} holds with $q$ replaced by  some $q-\epsilon$ and then invoke outer Marcinkiewicz interpolation (see Proposition~\ref{p.outerInterp}) to pass to the strong estimate at $q$ itself.  Moreover, it suffices to show \eqref{e.weakwembed-q}  holds for a Schwartz function $f$ on $\R$.  Indeed, given $f \in L^q(w)$ pick a sequence of Schwartz functions $f_k$ converging to $f$ in $L^q(w)$ norm satisfying
    \begin{itemize} 
        \item $\|f_1\|_{L^q(w)} \le C \|f\|_{L^q(w)}$ and 
        \item $\|f_{k+1}-f_{k}\|_{L^q(w)} \le C 2^{-10 k} \|f\|_{L^q(w)}$.
    \end{itemize} 
It is clear $P(f) = \lim_k P (f_k)$ pointwise and if we assume \eqref{e.weakwembed-q} holds for Schwartz functions,
    \[\|P(f_{k+1}) - P(f_k)\|_{\mathcal{L}^{q,\infty}(X,\sigma^w_\Delta,\s^w_\Delta)}  \lesssim \|f_{k+1}-f_k\|_{L^q(w)} \lesssim C 2^{-10 k} \|f\|_{L^q(w)}.\] 
The reduction follows by application of Proposition~\ref{p.outerDom} to the sequence $P(f_k)$. 

We prove the outer $L^\infty$ estimate \eqref{e.weakwembed-infty} in Section~\ref{subs.weakLinfty}  by appealing to Littlewood-Paley square function estimates to control each tent.  Section~\ref{subs.weakLq} handles the more complicated weak outer $L^{q}$ estimate \eqref{e.weakwembed-q}  by good-lambda type arguments restricted to tents with large size.  We remain in the discrete parameter setting for the rest of the paper, unless otherwise stated. As such, we drop the $\Delta$ notation and denote $\E = \E_\Delta$, $\s = \s^w_{\Delta}$, and $\mu = \mu^w_\Delta$.

\subsection{The outer $L^\infty$ embedding \eqref{e.weakwembed-infty}}\label{subs.weakLinfty}

Consider $f \in L^\infty(w)$. The goal is to show  
$$\s\big(P(f)\big)(T(x,\xi,s))  \le C\|f\|_{L^\infty(w)}$$
with implicit constant $C$ independent of  $(x,\xi,s)\in X_{\Delta}$ and $f$. As $\s$ is the sum of an $L^\infty$ size and an $L^2(w)$ size, we prove this inequality for each size separately. The estimate for the $L^\infty$ portion is clear from the definition of the wave packet transform. $$   \sup_{(y,\eta,t) \in T^b(x,\xi,s)} | \< f , \phi_{y,\eta,t}\>| \le \|f\|_{\infty} \|\phi\|_1 \le C \|f\|_{L^{\infty}(w)}. $$
    
It remains to control the $L^2(w)$ portion of size $\s$, 
    \begin{equation}\label{e.weakLinfty_s2}
    \int_{T^\ell(x,\xi,s)} |P(f)(y,\eta,t)|^{2} w(y-t,y+t)dy d\eta \frac{dt}{t} \le C w(x-s,x+s) \|f\|_\infty^{2} .
    \end{equation}
 Fix $(x,\xi,s) \in X_\Delta$ and consider the decomposition $f=f_1+f_2$ where $f_1 := f 1_{(x-2s,x+2s)}$.  For each $(y,\eta,t)$ such that  $y\in (x-s,x+s)$ and $t<s$, 
$$|P(f_2)(y,\eta,t)| \le \int_{[-s,s]^c} \big|f_2(y-z)  \frac 1 t \phi\Big(\frac z t\Big)\big| dz \le C \frac{t}{s}   \|f\|_\infty.$$
As $(y,\eta,t) \in T^\ell(x,\xi,s)$, we have the containment $(y-t,y+t) \subset (x-s,x+s)$. Therefore
$$ \left\|S_{T}\big(P(f_2)\big)\right\|_{L^2(w)}^2 = \int_{T^\ell(x,\xi,s)} |P(f_2)(y,\eta,t)|^{2} w(y-t, y+t)dy d\eta \frac{dt}{t} $$
$$\le C\int_0^s \int_{x-s}^{x+s} \int_{\xi - C_1 t^{-1}}^{\xi+C_2 t^{-1}} \Big(\frac{t}{s}   \|f\|_\infty \Big)^2  w(y-t, y+t) \, d\eta dy  \frac{dt}{t} \le C w(x-s,x+s) \|f\|_\infty^{2} $$
which establishes the desired estimate on $f_2$. 

For the $f_1$ piece, it suffices to show the square function estimate
\begin{equation}\label{e.squareTL2}
\left\|S_{T}\big(P (h) \big) \right\|_{L^q(w)}\le C_{\phi,q,[w]_{\A_{q/2}}}  \|h \|_{L^q(w)},
\end{equation}
for any $h \in L^q(w)$. As $S_{T}(P (f_1))$ is compactly supported, we can pass to \eqref{e.squareTL2} by applying H\"older's inequality to the left side of \eqref{e.weakLinfty_s2}. From there, appeal to the compact support of $f_1$ and doubling property of $w$ to control $\|f_1\|_{L^q(w)}$ by the desired result. 
By symmetry, we only need to establish \eqref{e.squareTL2} where $T^\ell$ is replaced by $T^\ell \cap \{\eta>\xi\}$. Consider a change in variables  $\eta=\xi+ \gamma /t$, and an absolute constant $C' > \max(C_1,C_2)$. It remains to show
$$\Big\| \Big(\int_0^\infty \int_{\R} \int_{b}^{C'} \big|P(h) (y,\xi+ \frac{\gamma}t, t) \big|^2   1_{|y-u|<t} d\gamma dy \frac{dt}{t^2} \Big)^{1/2} \Big\|_{L_u^q(w)} \le C \,  \|h\|_{L^q(w)}.$$ 
Let $g(u)=h(u) e^{-i \xi u}$ and  $M_\gamma \phi(z) = e^{i\gamma z}\phi(z)$. Since $b\le \gamma \le C'$ and $\widehat \phi$ is supported in $(-2^{-8}b,2^{-8}b)$, the frequency support of $M_\gamma \phi$ is  bounded away from $0$ and $\infty$ and  $M_\gamma \phi$ satisfies the usual decay estimates (where the implicit constant can be chosen uniformly over $b\le \gamma\le C'$). Observe that
$$ P (h)\big(y,\xi+ \frac{\gamma}t, t\big) = e^{i\xi y}  \Big(g* (M_\gamma \phi)_t(y)\Big)$$
where $(M_\gamma \phi)_t(z)= t^{-1} M_{\gamma}\phi(z/t)$.
Uniformly over $b \le \gamma \le C'$, we have 
	$$ \Big\| \Big(\int_0^\infty \int_{\R} \big|P (h)(y,\xi+\frac{\gamma}{t},t)\big|^{2}   1_{|y-u|<t}dy  \frac{dt}{t^2}\Big)^{1/2} \Big\|_{L_u^q(w)} = $$
	$$ = \Big\| \Big(\int_0^\infty \int_{\R} \big|g * (M_\gamma \phi)_t(y) \big|^{2}  1_{|y-u|<t} dy\frac{dt}{t^2}  \Big)^{1/2} \Big\|_{L_u^q(w)} \le C \|g\|_{L^q(w)}  = C\|h\|_{L^q(w)} $$
where the second-to-line inequality is a consequence of $L^q(w)$ boundedness for continuous square function estimates with $w \in \A_{q} \supset \A_{q/2}$; see \cite{lerner2011,lerner2014} for details. 
This establishes \eqref{e.squareTL2} and concludes the proof of the outer $L^\infty$ embedding \eqref{e.weakwembed-infty}.

\subsection{The weak outer $L^{q}$ embedding \eqref{e.weakwembed-q}}\label{subs.weakLq}

As mentioned, we may assume that $f$ is a Schwartz function on $\R$. Given $\lambda>0$ we need to find a countable collection of points $Q\subset X_\Delta$ such that
$$\sum_{(x,\xi,s)\in Q} w(x-s,x+s) \lesssim \lambda^{-q}\|f\|_{L^q(w)}^q,$$
and for every $T\in \E_\Delta$ we have, with $E:=\bigcup_{(x,\xi,s)\in Q} T(x,\xi,s)$,
\begin{equation}\label{e.weakLqsize}
\s\big( P(f)1_{X\setminus E}\big)(T)\le \lambda.    
\end{equation}   

We first reduce to the case when $\widehat f$ is compactly supported.  Indeed, we may select frequencies $\xi_0=0<\xi_1<\xi_1<\dots$  such that with $\widehat f_k=\widehat f1_{\xi_{k-1}\le |\xi|\le \xi_k}$, $k\ge 1$, satisfy
$$\|f_k\|_{L^q(w)} \le C 2^{-10k} \|f\|_{L^q(w)}.$$
Using the special case to each $f_k$ with $\lambda_k=2^{-k}\lambda$ we obtain collections $Q_k$, and clearly
$$ \sum_{k\ge 1} \sum_{(x,\xi,s)\in Q_k} w(x-s,x+s) \le C \lambda^{-q}   \sum_k  2^{qk}  \|f_k\|_{L^q(w)}^q \le C\lambda^{-q}\|f\|_{L^q(w)}^q. $$
On the other hand using subadditivity of the size, we may estimate
$$\s\big( P(f)1_{X\setminus E}\big)(T)\le \sum_{k\ge 1} \s\big(P(f_k)1_{X\setminus E}\big)(T) \le  \sum_{k\ge 1} 2^{-k}\lambda \le \lambda.$$
Thus from now on we may assume that $f$ is Schwartz such that $\widehat f$ is compactly supported.

\subsubsection{Treatment for the $L^\infty$ part of the size}\label{ssub:Linfsize}
 
We first isolate tents
$T \in \E_\Delta$ which contain all $(y,\eta,t) \in X$ such that $|P(f)(y,\eta,t)| > \lambda$. Note that by Cauchy-Schwarz,
\begin{equation}\label{e.Poisson_ptBd}
    |P(f)(y,\eta,t)| = \Big|\int_\R f(y-z) e^{i\eta z}\frac 1 t \phi\big(\frac  z t\big) dz \Big| \le  t^{-1/2} \|f\|_2 \|\phi\|_2  = O(t^{-1/2}),
\end{equation}
so if $|P(f)(y,\eta,t)|>\lambda$ then there is an a priori bound $t<O(\lambda^{-2})$. We remark this estimate is not essential; the upper bound is only needed to ensure that the $L^\infty$ selection algorithm described below will terminate and it will  never be used quantitatively. 

{\bf $\bm{L^\infty}$ Selection Algorithm}. Suppose there is some $(y,\eta,t) \in X$ such that $|P(f)(y,\eta,t)| > \lambda$.
By Lemma~\ref{l.centralcontain}, we can associate with $(y,\eta,t)$ a point $(x,\xi,s) \in X_\Delta$ such that $(y,\eta,t)$ is centrally contained in $T(x,\xi,s)$. In particular $s$ has to be bounded a priori by \eqref{e.Poisson_ptBd} and is always a power of $2$. We may therefore choose  $(y_1,\eta_1,t_1) \in X$ and $(x_1,\xi_1,s_1) \in X_\Delta$ so that  
 \begin{itemize}\itemsep0.15em
 \item $|P(f)(y_1,\eta_1,t_1)|>\lambda$, 
 \item $(y_1,\eta_1,t_1) \in T(x_1,\xi_1,s_1)$ centrally, and
 \item  $s_1$ is maximal.
 \end{itemize}
We iterate this process. Assume that we have already selected  $(y_k,\eta_k,t_k) \in X$ and $(x_k,\xi_k,s_k) \in X_\Delta$ for $1 \le k < n$. Suppose there is a point $(y,\eta,t) \in X$ outside the union of selected tents $T(x_k,\xi_k,s_k)$ satisfying $|P(f)(y,\eta,t)|>\lambda$. We now choose $(y_n, \eta_n, t_n) \in X$ and $(x_n,\xi_n,s_n) \in X_\Delta$ such that
\begin{itemize} \itemsep0.15em
    \item $(y_n,\eta_n,t_n)  \not\in \bigcup_{k=1}^{n-1} T(x_k,\xi_k,s_k)$,
    \item $|P(f) (y_n,\eta_n,t_n)|>\lambda$, 
    \item $(y_n,\eta_n,t_n) \in T(x_n,\xi_n,s_n)$ centrally, and 
    \item $s_n$ is maximal. 
\end{itemize}

Figure~\ref{f.type1sep} gives a visual representation of the selection.  For each $k \ge 1$, denote $T_k = T(x_k,\xi_k,s_k)$ and $I_k=(x_k-s_k, x_k+ s_k)$ for notation conveinence.
 
 \begin{figure}
 \begin{center}
 \includegraphics[scale=0.4]{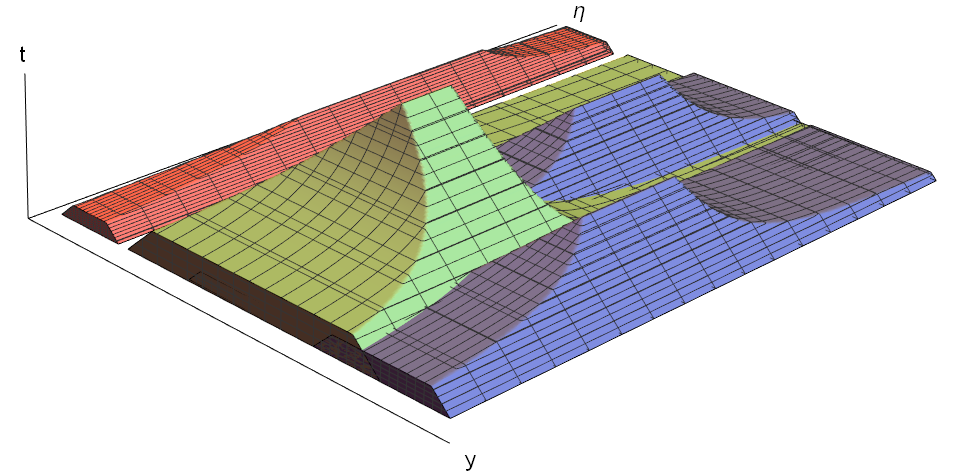}
 \caption{Example of tents from the $L^\infty$ selection algorithm. Here, the green tent (with largest scale) is selected first, followed by the two blue tents and finally the red tent (with smallest scale).}
 \label{f.type1sep}
 \end{center}
 \end{figure}
 
 Our goal is to show that
\begin{equation}\label{e.Linfty}
\sum_{k=1}^n w(I_k) \le C\lambda^{-q} \|f\|_{L^q(w)}
\end{equation}
for all $n$. Assuming \eqref{e.Linfty} holds, we justify the algorithm successfully terminates. If the algorithm finishes after $n$ steps, then all points $(y,\eta,t)$ outside the union of tents $T_k$ for $1 \le k \le n$ satisfy $|P(f)(y,\eta,t)|\le \lambda$. Suppose the algorithm doesn't terminate after a finite number of steps. We  observe the sequence of selected  heights $s_k$ must decay  to $0$ in this scenario. This is due to the selection process being independent of the chosen weight $w$ meaning \eqref{e.Linfty} would hold for the Lebesgue case $w=1$. 
Therefore, if we consider some  $(y,\eta,t)$ outside the union $\bigcup_{k \ge 1} T_k$  then $t > s_j$ for some $j$. As $T_j$ is the tallest tent at the $j$-step with respect to containing a centralized point whose wavelet projection is greater than $\lambda$, we conclude $|P(f)(y,\eta,t)|\le \lambda$.  Thus, up to the proof of \eqref{e.Linfty}, all points $(y,\eta,t)$ outside the union of the selected tents in the algorithm must satisfy $|P(f)(y,\eta,t)|\le \lambda$. 

It remains to prove weighted estimate \eqref{e.Linfty}.  We first make the following observation regarding the well separation of the selected points $(y_k,\eta_k,t_k)$. 
 \begin{claim}\label{cl.SAlg_Linfty_WellSep}
 The collection of points  $\{(y_k,\eta_k,t_k)\}_{k \ge 1}$ is well separated in the sense of \eqref{e.wellsep_pts} with separation constants $\alpha = 2^{-2}$ and $\beta = 2^{-6}b$; here, $b$ is the parameter in $\Theta = (C_1,C_2,b)$.
 \end{claim}
 
To verify the claim, consider $(y_j,\eta_j,t_j)$, $(y_k,\eta_k,t_k)$ such that $j < k$. By the selection algorithm,  $(y_j,\eta_j,t_j) \in T_j$ was selected prior to $(y_k,\eta_k,t_k) \in T_k$.  Suppose to the contrary that 
    $$|\eta_j - \eta_k| \le 2^{-6} b \max(t_j^{-1},t_k^{-1}) \quad \text{ and } \quad |y_j - y_k| \le 2^{-2} \max(t_j,t_k).$$
Central containment and $s_j \ge s_k$ implies $|y_j - y_k| \le 2^{-4} s_j$ and $|\eta_j - \eta_k| \le 2^{-3}b s_k^{-1}$. Thus,
	$$|y_k- x_j| \le |y_k - y_j| + |y_j - x_j| \le 2^{-3} s_j \le  s_j - t_k$$ 
and
	$$ \eta_k - \xi_j = (\eta_k - \eta_j) + (\eta_j - \xi_j) \le t_k^{-1} 2^{-4}b \le t_k^{-1}C_2 $$
with similar work gives $\eta_k -\xi_j \ge -t_k^{-1}C_1$. As $t_k < s_k \le s_j$, this means $(y_k,\eta_k,t_k) \in T_j$ which contradicts the assumption the point was selected after $T_j$ was removed. Thus, Claim~\ref{cl.SAlg_Linfty_WellSep} is true.

By the three grids trick (see Remark~\ref{r.3grids}),  each $I_k$ is contained in a dyadic interval $J_k$ of comparable length where $J_k$ is either in the standard dyadic grid $\mathcal{D}_0$,  the $1/3$ shifted grid $\mathcal{D}_1$, or the $2/3$ shifted grid $\mathcal{D}_2$. It suffices to prove 
\begin{equation}\label{e.Linfty_dyadic}
\sum_{k =1}^n w(J_k) \le C\lambda^{-q} \|f\|_{L^q(w)}
\end{equation}
where $J_k$ is an element of $\mathcal{D}_0$, $\mathcal{D}_1$, or $\mathcal{D}_2$  such that $I_k \subset J_k$ and $3 |I_k| \le |J_k| \le 6 |I_k|$. Without loss of generality, we may assume all $J_k$ belong to the same dyadic grid, which for convenience we assume to be the standard grid $\mathcal{D}_0$. 

For each $m\ge 0$ let $K_m$ be the set of $k$ such that $2^m \lambda \le |P(f)(y_k,\eta_k, t_k)| \le 2^{m+1}\lambda$. It suffices to show that for any $m$
$$\sum_{k\in K_m} w(J_k)  \le C (2^m\lambda)^{-q } \|f\|_{L^q(w)}^q.$$
We show this for $m=0$ as the general case can be obtained by simply letting $\lambda'=2^m\lambda$ and repeating the argument.  Fixing $m=0$, assume without loss of generality that $K=K_0$ so all $k$ are automatically inside $K_0$.

For convenience let
    $$N=\sum_k 1_{J_k}  \qquad \quad \text{ and } \qquad  N_I = \sum_{k: J_k\subset I} 1_{J_k} $$
be the tent counting function and counting function restricted to interval $I$ respectively. The following lemma is a localization inequality for the $L^1(w)$ norm of $N_I$.
 
\begin{lemma}\label{l.L1byLinfty} Fix a  dyadic interval $I$, $\alpha>0$, and integer $n \ge 0$. Let $\chi_I(x)$ be as defined in \eqref{e.localization}. Then
    $$\lambda \|N_I\|_{L^1(w)}^{1/q} \lesssim_{n,\alpha} \|N_I\|_{\infty}^\alpha \|f  \chi_I^n\|_{L^q(w)}.$$
\end{lemma}
\proof Fix interval $I$, without loss of generality we may assume $N=N_I$. We first consider the simpler proof for $n=0$. For convenience, denote  $P(f)(y_k,\eta_k,t_k) = \<f,\phi_k\>$ where 
$$\phi_k(x)\equiv \phi_{y_k,\eta_k,t_k}(x)  = \frac 1 {t_k}e^{-i\eta_k(y_k-x)} \overline{\phi\Big(\frac{y_k-x}{t_k}\Big)}.$$
As $\lambda \le |P(f)(y_k,\eta_k,t_k)| \le 2 \lambda$, we obtain
$$\lambda^2 \sum_k w(J_k) \lesssim \sum_{k\in K_m} w(J_k) |P(f)(y_k,\eta_k,t_k)|^2 = \|S (f)\|_{L^2(w)}^2$$
where 
$$S (f)(x):=\Big(\sum_{k} |\<f,\phi_k\>|^2 1_{J_k}(x) \Big)^{1/2}$$
is the square function summed over the selected points. 

We now implement a sharp maximal inequality argument. Note first that $supp\big(S(f)\big)\subset supp(N)$. Using H\"older's inequality and the sharp maximal inequality we have
$$  \big\|S(f)\big\|_{L^2(w)}  \le w\big(supp(N)\big)^{\frac{1}{2}-\frac{1}{q}} \big\|S(f)\big\|_{L^q(w)} \lesssim \Big(\sum_k w(J_k)\Big)^{\frac{1}{2}-\frac{1}{q}} \big\|(S(f))^{\#}\big\|_{L^q(w)} $$
where $(S(f))^{\#}$ is the dyadic sharp maximal function, taken  with respect to the  grid that contains all intervals $J_k$. Consequently,

\begin{equation}\label{e.Linfty_est1}
    \lambda \Big[\sum_k w(J_k)\Big]^{1/q}  \lesssim   \big\|(S(f))^{\#}\big\|_{L^q(w)}.
\end{equation}

We will now establish the pointwise estimate
\begin{equation}\label{e.sharpLinfty}
(S(f))^{\#}(x) \lesssim_{\phi,s} M_2(f)(x) + \lambda^s    [M(N)(x)]^{s/2}    [M_2(f)(x)]^{1-s}
\end{equation}
where $M$ is the usual dyadic maximal function and $s \in (0,1)$. For each dyadic interval $J$ it suffices to show there is some constant $\alpha_J>0$ such that
\begin{equation*}
\left(\frac{1}{|J|} \int_J |S(f)(x)-\alpha_J|^2 dx \right)^{1/2} \lesssim \inf_{x\in J} M_2(f)(x) + \lambda^s  \inf_{x\in J} [M(N)(x)]^{s/2}  \inf_{x\in J} [M_2(f)(x)]^{1-s}
\end{equation*}
Setting $c_J$ as the center of $J$, let $\alpha_J = S_{out} (f)(c_J)$ where 
$$S_{out}(f)(x):=\Big(\sum_{k: J\subset J_k} |\<f,\phi_k\>|^2 1_{J_k}(x)\Big)^{1/2}.$$
For convenience, let $S_{in} (f)(x)=(S (f)(x)^2 - S_{out}(f)(x)^2)^{1/2}$ and consider  $f_J = f  \chi_J^N$ for some large constant $N$. 
By the square function estimate (from Lebesgue theory) we have
$$\int_J |S(f)(x)-\alpha_J|^2 dx \le \int_J |S_{in}(f)(x)|^2dx = \int_J |\widetilde S_{in} (f_J)(x)|^2dx $$
where $\widetilde S_{in}$ uses the mollified wave functions $\chi_J^{-N}\phi_{k}$. Note the mollified wave function has the same frequency support as  $\phi_k$ and has sufficient decay while localization over $(y_k-t_k,y_k+t_k)$ in the sense Lemma~\ref{l.wavepacket1} is applicable for appropriate exponents. Recall from Claim~\ref{cl.SAlg_Linfty_WellSep} that the selected points $(y_k,\eta_k,t_k)$ are well-separated with separation constants dependent solely on $\Theta$. We can therefore apply Lemma~\ref{l.wellsep1} along with the observation $t_k \le |J_k|$ at each $k$ to get
\begin{align*}
    \left(\int_J |S(f)(x)-\alpha_J|^2 dx \right)^{1/2} &\lesssim  \| f_J \|_2   +  \Big[\sup_{k}  \big|P(f)(y_k,\eta_k,t_k)\big| \Big(\sum_{k: J_k\subset J} |J_k|\Big)^{1/2} \Big]^{\, s} \|f_J\|_2^{1-s} \\
        &\lesssim |J|^{1/2} \inf_{x\in J} M_2(f)(x) + \lambda^s |J|^{1/2} \inf_{x\in J} [M(N)(x)]^{\, s/2}  \inf_{x\in J} [M_2(f)(x)]^{1-s}
\end{align*} 
as desired.

Now, using  \eqref{e.sharpLinfty} and the fact that $w\in \A_{q/2}$, 
\begin{align*}
    \big\|(S(f))^{\#}\big\|_{L^q(w)} 
    &\lesssim \|f\|_{L^q(w)} +  \lambda^{s} \|N\|_{L^{q/2}(w)}^{s/2} \|f\|_{L^q(w)}^{1-s} \\
    &\lesssim \|f\|_{L^q(w)} + \lambda^s \|N\|_{\infty}^{(\frac{s}{q})(\frac{q}{2}-1)}   \|N\|_{L^1(w)}^{s/q}   \|f\|_{L^q(w)}^{1-s}.
\end{align*}
As $\|N\|_{L^1(w)} = \sum_k w(J_k)$, combine \eqref{e.Linfty_est1} with the above sharp maximal function bounds to get
$$\lambda \|N\|_{L^1(w)}^{1/q} \lesssim  \|N\|_{\infty}^{\frac{s}{1-s}(\frac 12 -\frac 1 q)}\|f\|_{L^q(w)}.$$
By selecting $0<s<1$ suitably we obtain the desired conclusion for any $\alpha>0$, but for $n=0$.
For $n>0$,   apply the argument above for the mollified wave functions $\widetilde \phi_k(x) =  \chi_I(x)^{-n}\phi_k(x)$, and note that $\<f , \phi_k\> = \<f \chi_I^n ,\widetilde \phi_k\>$ for all $k$. As previously stated, the mollified wave function $\widetilde{\phi}_k$ has the same frequency support as  $\phi_k$ and   sufficient decay while still localized over $(y_k-t_k,y_k+t_k)$ so the same analysis as before applies. 
\endproof

With the  $L^1(w)$ norm estimate on $N_I$, we deduce a similar estimate for $L^{r,\infty}(w)$ with $0<r<1$.

\begin{corollary}\label{c.LrbyLinfty} Fix dyadic interval $I$. For any $r \in (0,1)$ and $n>0$ it holds that
$$\|N_I\|_{L^{r,\infty}(w)} \lesssim (\lambda^{-1}\|f \chi_I^n \|_{L^q(w)})^{q/r}.$$
\end{corollary}

\proof Arguing as before, without loss of generality we may assume that $N=N_I$ and $n=0$. Given any $t>0$, we have
$$w(N>t) = \sum_{k\ge 0} w(2^k t<N\le 2^{k+1}t).$$

For each $k$, let $A_k$ be the collection of $k$ such that $J_k \cap \{N\le 2^{k+1}t\} \ne \emptyset$ and denote $N_k$ as the counting function $N$ restricted to $A_k$. 
Then for every $x\in  \{N\le 2^{k+1}t\}$ we have $N(x) = N_k(x)$, therefore
$$\{2^k t<N\le 2^{k+1}t\} \subset \{N_k > 2^{k}t\}.$$
Furthermore, we also have
$$\|N_k\|_\infty \le 2^{k+1}t.$$
Indeed, $N_k$ is locally constant, and any interval on which $N_k$ is constant must be part of some $J_k$ that in turn intersects $\{N\le 2^{k+1}t\}$, and clearly $N_k\le N$.

Thus, applying Lemma~\ref{l.L1byLinfty} with $\alpha>0$ we obtain 
$$\lambda \|N_k\|_{L^1(w)}^{1/q} \lesssim \|N_k\|_\infty^\alpha \|f\|_{L^q(w)} \lesssim 2^{(k+1)\alpha}t^\alpha \|f\|_{L^q(w)}$$
therefore
$$ \lambda^q w(N_k>2^{k+1}t) \le 2^{-(k+1)}t^{-1} \lambda^q \|N_k\|_{L^1(w)}   \lesssim  2^{-(k+1)(1-\alpha q)} t^{-(1-\alpha q)}\|f\|_{L^q(w)}^q$$
Summing over $k\ge 0$, it follows that
$$\lambda^q w(N>t) \lesssim t^{-(1-\alpha q)} \|f\|_{L^q(w)}^q$$
provided that $0<\alpha < 1/q$. Letting $r=1-\alpha q$ we obtain
$$tw(N>t)^{1/r} \lesssim \lambda^{-q/r} \|f\|_{L^q(w)}^{q/r}.$$
To get $n>0$, apply the argument above for mollified wave functions $\widetilde \phi_k = \chi_I^{-n}\phi_k$. 
\endproof

Next, we establish a good $\lambda$ inequality with respect to $N$. Below,   let $M_{q,w}f$ be the weighted $L^q$-maximal function.

\begin{lemma}\label{l.goodlambda}
Consider $L>0$ and $r\in (0,1)$. There is some $c>0$ such that for any $t>0$,
$$w(N>t) \le \frac 1 L w (N>t/4) + w(M_{q,w}f> c\lambda t^{r/q}) .$$
\end{lemma}
\proof Let $\bf I$ be the collection of all maximal dyadic intervals that are subsets of $\{N>t/4\}$. Suppose that $I\in \bf I$ and $I\cap \{M_{q,w}f \le c\lambda t^{q/r}\} \ne\emptyset$. Then  for every $x\in \{N > t\}\cap I$ we have $N(x)-N_I(x)\le  t/4$, so in particular $x\in  \{N_I>t/4\}$. It follows that, using Corollary~\ref{c.LrbyLinfty},
$$w(\{N>t\}\cap I) \le w(N_I>t/4) \lesssim t^{-r} \lambda^{-q}\|f \chi_I^n\|_{L^q(w)}^q$$
$$\lesssim t^{-r} \lambda^{-q} w(I) \inf_{x\in I} M_{q,w}(f)(x)^q \lesssim t^{-r} \lambda^{-q} w(I) (c\lambda t^{r/q})^q \lesssim c^q w(I).$$
Consequently, by choosing $c>0$ sufficiently small (independent of $I$) we obtain
$$w(\{N>t\}\cap I) \le \frac 1 L w(I).$$
Summing over $I$ we obtain the desired claim.
\endproof

We are now ready to show \eqref{e.Linfty_dyadic}. Integrating over $t>0$ in the good $\lambda$ estimate provided by Lemma~\ref{l.goodlambda}, it follows (from the standard argument) that
	$$ \|N\|_{L^1(w)} \lesssim \int_0^\infty w(M_{q,w} (f)>c\lambda t^{r/q})dt =(c\lambda)^{-q/r}\int_0^\infty w(M_{q,w} (f)>s) s^{\frac qr -1}ds $$
	$$ \lesssim \lambda^{-q/r} \|M_{q,w}(f)\|_{L^{q/r}(w)} \lesssim \lambda^{-q/r} \|f\|_{L^{q/r}(w)} .$$
Using the fact that the  $\A_q$ condition is an open condition, we actually have $w\in \A_{qr}$ for some $r\in (0,1)$. Thus, by replacing $q$ with $qr$, 
$$\|N\|_{L^1(w)} \lesssim \lambda^{-q} \|f\|_{L^q(w)}.$$
This completes the proof of \eqref{e.Linfty_dyadic} which in turn implies the desired estimate \eqref{e.Linfty}.

We finish by denoting $Q_0$ as the collection of $(x_k,\xi_k,s_k)$ selected and $$E_0 = \bigcup_{(x,\xi,s)\in Q_0} T(x,\xi,s).$$  We have $|P(f)(y,\eta,t)|\le \lambda$ for all $(y,\eta,t)\in E_0^c$, while the weighted outer measure of $E_0$ is $O(\lambda^{-q}\|f\|_{L^q(w)})$.  We free the notations $(y_k,\eta_k,t_k)$, $(x_k,\xi_k,s_k)$, $T_k$, $I_k$, $k\ge 1$.

\subsubsection{Treatment for the $L^2$ part of the size} We now select tents over which the $L^2(w)$ portion of size $\s$ is large with respect to $\lambda$. We split a tent $T(x,\xi,s)$ into its upper half 
	$$T_+(x,\xi,s) = T(x,\xi,s) \cap \{ (y,\eta,t) \in X \, : \,  \eta \ge \xi\}$$
and its lower half $T_-(x,\xi,s) = T(x,\xi,s) \backslash T_+(x,\xi,s)$.  We define $T^b_{\pm}$ and $T^\ell_{\pm}$ similarly.

This section focuses on finding a countable collection of points $Q_+ \subset X_\Delta$ such that
    $$ \sum_{(x,\xi,s) \in Q_+} w(x-s,x+s) \le C \lambda^{-q} \|f\|_{L^q(w)}$$
and for every tent  $T(x,\xi,s) \in \E_\Delta$, 
    $$ \frac 1 {w(x-s,x+s)}\int_{T^\ell_+(x,\xi,s) \backslash E_+} |P(f)(y,\eta,t)|^2 w(y-t,y+t)dy d\eta \frac{dt}t \le \lambda^2 $$
where 
    $$ E_+ = E_0 \cup \bigcup_{(x.\xi,s) \in Q_+} T(x,\xi,s). $$
We remark that the work to generate $Q_+$ and $E_+$ can be applied symmetrically to find a countable collection of points $Q_- \subset X_\Delta$ such that
    $$ \sum_{(x,\xi,s) \in Q_-} w(x-s,x+s)  \le C \lambda^{-q} \|f\|_{L^q(w)}$$
and if 
    $$ E_- = E_0 \cup  \bigcup_{(x,\xi,s) \in Q_-} T(x,\xi,s) , $$
then 
    $$ \frac 1 {w(x-s,x+s)}\int_{T^\ell_-(x,\xi,s) } |P(f)(y,\eta,t) 1_{X \backslash E_-}(y,\eta,t)|^2 w(y-t,y+t)dy d\eta \frac{dt}t \le  \lambda^2 $$
for all tents $T(x,\xi,s) \in \E_\Delta$. By setting  $Q = Q_0 \cup Q_+ \cup Q_-$ and $E:=\bigcup_{(x,\xi,s)\in Q} T(x,\xi,s)$,  the proof of the $q$-endpoint estimate \eqref{e.weakwembed-q} will then be completed.

{\bf $\bm{L^2}$ Selection Algorithm}. Let $C$ be larger than the doubling constant of $w$. Suppose there exists $(x,\xi,s)\in X_\Delta$ such that
\begin{equation}\label{e.L2selectalg_badpt}
     \frac 1 {w(x-s,x+s)}\int_{T^\ell_+(x,\xi,s)  \setminus  E_0} |P(f)(y,\eta,t)|^2 w(y-t,y+t)dy d\eta \frac{dt}t \ge C^{-1}\lambda^2 . 
\end{equation}
Applying \eqref{e.squareTL2}, it follows that $w(x-s,x+s)$  is bounded from above a priori. By substituting $f=f \chi^N_{(x-s,x+s)}\chi^{-N}_{(x-s,x+s)}$ in \eqref{e.L2selectalg_badpt} where $N>q$, similar work in conjunction with Remark~\ref{r.ApNotL1} shows  $s$ itself is bounded  a priori. Indeed, as $w \not\in L^1$ is a doubling weight, it forces $|x|/s$ to be sufficiently large whenever $s$ itself is large. Consequently, $\|  f \chi^N_{(x-s,x+s)} \|_\infty$ is sufficiently small. As 
    $$\| \chi^N_{(x-s,x+s)}\|_{L^1(w)}\lesssim w(x-s,x+s),$$
we conclude the value of $s$ for points $(x,\xi,s)$ satisfying \eqref{e.L2selectalg_badpt} must be bounded.

Let $\widetilde{s}$ be the least upper bound on $s$. As such, $\xi$ is a discrete parameter since it is a multiple of $2^{-8}b (\,  \widetilde{s} \, )^{-1}$. Since $\widehat{f}$ is compactly supported, there is an upper bound on $\xi$
meaning there is a maximal possible value $\xi = \xi_{\max}$ to consider in \eqref{e.squareTL2}. Select $(x_1,\xi_1, s_1) \in X_\Delta$ satisfying \eqref{e.L2selectalg_badpt} such that $\xi_1=\xi_{\max}$  and $s_1$ is maximal with respect to the restriction $\xi_1 = \xi_{\max}$. Let $T_1 = T(x_1,\xi_1,s_1)$ and  $I_1=(x_1-s_1,x_1+s_1)$  for convenience of notation. 
By the maximality of $s_1$ and the doubling property of $w$, note
$$\frac 1 {w(I_1)}\int_{T^\ell_+(x_1,\xi_1,s_1)  \setminus  E_0} |P(f)(y,\eta,t)|^2 w(y-t,y+t)dy d\eta \frac{dt}{t} \le \lambda^2.$$
We now iterate the argument. Assume that we have selected $(x_k,\xi_k,s_k) \in X_\Delta$ for $1\le k \le n-1$ and set $E_{n}=E_0\cup \bigcup_{k=1}^{n-1} T_k$. 
Suppose there is some $(x,\xi,s) \in X_{\Delta}$ such that 
$$\frac 1 {w(x-s,x+s)}\int_{T^\ell_+(x,\xi,s)  \setminus  E_{n}} |P(f)(y,\eta,t)|^2 w(y-t,y+t)dy d\eta \frac{dt}t \ge C^{-1}\lambda^ .2$$
We now select such a point  $(x_n,\xi_n,s_n)$ such that $\xi_n$ is a (possibly new) maximal $\xi_{\max}$ and $s_n$ is maximized with respect to $\xi_{\max}$. Denote $T_n = T(x_n,\xi_n,s_n)$ and $I_n = (x_n-s_n,x_n+s_n)$.   Again,  by the maximality of $s_n$,
$$\frac 1 {w(I_n)} \int_{T^\ell_+(x_n,\xi_n,s_n)  \setminus E_{n}} |P(f)(y,\eta,t)|^2 w(y-t,y+t)dy d\eta \frac{dt}t \le \lambda^2.$$
 
Our goal is to show that
\begin{equation}\label{e.L2}
\sum_{k =1}^n w(I_k)  \le C \lambda^{-q}\|f\|_{L^q(w)}^q.
\end{equation}
where $C$ is independent of $n$. Assuming  \eqref{e.L2} is valid, we now justify the termination of the $L^2$ selection algorithm. If the algorithm terminates after selecting $n$ tents, then
    $$ \frac 1 {w(x-s,x+s)} \int_{T^\ell_+(x,\xi,s)  \setminus  E_{n+1}} |P(f)(y,\eta,t)|^2 w(y-t,y+t)dy d\eta \frac{dt}t \le  C^{-1} \lambda^2. $$
holds for all $(x,\xi,s) \in \E_\Delta$ and we set $Q_+$ as the collection of $(x_k,\xi_k,s_k)$ for $1 \le k \le n$. In the scenario the algorithm does not terminate after a finite number of steps, set  $E_{(1)}=E_0 \cup \bigcup_{k \ge 1} T_k$ and $Q_{+,(1)}$ as the collection of selected $(x_k,\xi_k,s_k)$ for $k \ge 1$. Note the selected $\xi_k$ form a non-increasing sequence of elements in the lattice $\Z 2^{-8}b  (\,  \widetilde{s} \, )^{-1}$. In the case $\xi_k$ tends to negative infinity, suppose there is some $(x,\xi,s) \in X_\Delta$ satisfying
\begin{equation}\label{e.L2badpt_E1}
    \frac 1 {w(x-s,x+s)}\int_{T^\ell_+(x,\xi,s)  \setminus E_{(1)}} |P(f)(y,\eta,t)|^2 w(y-t,y+t)dy d\eta \frac{dt}{t} \ge C^{-1}\lambda^2  .
\end{equation}
As $\xi_k \to -\infty$, then $\xi_j < \xi$ for some $j$ which contradicts the selection of $T_j$. 
Thus, the converse inequality to \eqref{e.L2badpt_E1} must hold for all $(x,\xi,s) \in X_\Delta$ and we set $Q_+ = Q_{+,(1)}$. 

Now suppose $\xi_k$ does not tend to negative infinity and instead stabilizes at some finite $\xi_{(1)}$.  We restart the algorithm and redefine the selected tents $T_k$ by $T_{(1),k} = T(x_{(1),k},\xi_{(1),k},s_{(1),k})$ and intervals $I_{k}$ by $I_{(1),k}$. Observe in this scenario that the tail of the sequence $s_{(1),k}$ decays to $0$. Indeed, if $s_{(1),k}$ converges to some nonzero $s_{(1)}$ then the tail of $Q_{+,(1)}$ is of the form $(x_{(1),k},\xi_{(1)},s_{(1)})$ where $x_{(1),k}$ is an element in the lattice $\Z 2^{-4} s_{(1)}$.  The corresponding intervals $I_{(1),k}$ therefore eventually slide  towards infinity or negative infinity. By recycling the argument for the  a priori bound on  $s$, such a sequence cannot occur.

Consider $(x,\xi,s) \in X_\Delta$ where \eqref{e.L2badpt_E1} holds. As before, there is a maximal frequency $\xi_{\max}$ to consider for such a point. Given the previous  selection of tents, we know $\xi_{\max} \le \xi_{(1)}$. If $\xi_{\max} = \xi_{(1)}$ and $(x,\xi_{\max},s)$ satisfies \eqref{e.L2badpt_E1} then $s > s_{(1),k}$ for some $k$  by the previous paragraph which contradicts the selection of $T_{(1),k}.$
Therefore, it follows that $\xi_{\max} < \xi_{(1)}$. Choose a point $(x_{(2),1} \xi_{(2),1}, s_{(2),1})$ satisfying \eqref{e.L2badpt_E1} such that  $s_{(2),1}$ is maximized under the condition $\xi_{(2),1} = \xi_{\max}$. Iterate the selection algorithm as before to obtain a sequence of tents and intervals 
    $$ T_{(2),k} = T(x_{(2),k},\xi_{(2),k},s_{(2),k}) \, , \qquad \text{ and } \qquad I_{(2),k} = (x_{(2),k} - s_{(2),k}, x_{(2),k}+s_{(2),k}).$$ 
 The proof of \eqref{e.L2}, to be shown, will naturally extend here to give
    $$ \sum_{k=1}^\infty w(I_{(1),k}) + \sum_{k \ge 1} w(I_{(2),k}) \le C \lambda^{-q} \|f \|_{L^q(w)}^q .$$
    
Continue the process as shown above. If we eventually have a sequence of frequencies $\xi_{(m),k}$ ($m$ is fixed) which either terminates after finitely many $k$ or $\xi_{(m),k} \to - \infty$ then there are no more $(x,\xi,s) \in X_\Delta$ for which the (now updated) version of \eqref{e.L2badpt_E1} holds. At worst, we  eventually obtain a double sequence of tents $T_{(m),k}$ 
where the double summation of $w(I_{(m),k})$  is $O(\lambda^{-q} \| f \|_{L^q(w)}^q)$. 
In addition, the sequence of stabilizing points $\xi_{(k)}$ in this case is strictly decreasing in a discrete lattice and tending to negative infinity. We finish by setting $Q_+$ as the collection of $(x_{(j),k},\xi_{(j),k},s_{(j),k})$ where $j,k \in \Nat$  and $E_+$ as the union of tents $T(x_{(j),k},\xi_{(j),k},s_{(j),k})$. We can therefore conclude that there are no more points $(x,\xi,s) \in X_\Delta$ such that 
    $$ \frac{1}{w(x-s,x+s)} \int_{T^\ell_+(x,\xi,s) \backslash E_+} |P(f)(y,\eta,t)|^2 w(y-t,y+t) \, dy d\eta \frac{dt}{t} \ge C^{-1} \lambda^2$$
and so the $L^2$ algorithm terminates. 

It remains to prove the weighted estimate \eqref{e.L2}. Note the proof  follows similar steps as in the $L^\infty$ treatment in Section~\ref{ssub:Linfsize}. For convenience of notation let $$T_k^* = T^\ell_+(x_k,\xi_k, s_k) \setminus  E_k \qquad \text{ for }k \ge 1$$ with $E_1 = E_0$. We first note the well separation of the  partial tents $T^*_k$.

\begin{claim}\label{cl.SAlg_L2_WellSep}
The collection of partial tents $\{T^*_k\}$ is well-separated in the sense of \eqref{e.wellsep_partialtents} with separation constants $\alpha=1$, $\beta = 2^{-6}b$, and $B= 2^8 \max(C_1,C_2)b^{-1}$  in terms of $\Theta = (C_1,C_2,b)$. 
\end{claim}
To verify the claim, consider $(y,\eta,t) \in T^*_j$ and $(y',\eta',t') \in T^*_k$ such that $B t' < t$. Assuming $|\eta - \eta'| \le 2^{-6}b t'^{-1}$,  it follows from $(y',\eta',t')$ being in the upper lacunary part of $T_k$ that $\xi_j> \xi_k$. 
	$$ \xi_j - \xi_k = (\eta - \eta') - (\eta - \xi_j) + (\eta' - \xi_k)  $$
	$$ \ge -2^{-6}b t'^{-1} - C_2 t^{-1}  + b t'^{-1}  \ge (1 - 2^{-6}  -2^{-8}) b t'^{-1} > 0  $$
	This means tent $T_j$ was selected prior to $T_k$  so $(y',\eta',t') \not\in T_j$ by the selection process. Furthermore, observe that $t' < t < s_j$ and  
	$$ \eta' - \xi_j = (\eta'-\eta)+(\eta-\xi_j) \le \big( 2^{-7}b + C_2 B^{-1}\big)  (t')^{-1} \le C_2 (t')^{-1}$$
    with similar work showing $\eta' - \xi_j \ge -C_1 (t')^{-1}$.  As $(y',\eta',t') \not\in T_j$, we then require
        $$|y'-x_j| > s_j - t'> s_j - t$$
    which verifies Claim~\ref{cl.SAlg_L2_WellSep}. 
    
As in the $L^\infty$ selection argument, application of the three grids trick (see Remark~\ref{r.3grids})  means it suffices to show 
\begin{equation}\label{e.L2_dyadic}
    \sum_k w(J_k) \le C \lambda^{-q} \|f\|_{L^q(w)}^q    
\end{equation} 
where all $J_k$ are dyadic intervals in the standard grid $\mathcal{D}_0$, $1/3$ shifted grid $\mathcal{D}_1$, or $2/3$ shifted grid $\mathcal{D}_2$   such that $I_k \subset J_k$ and $3 |I_k| \le |J_k| \le 6 |I_k|$. We may assume without loss of generality that all $J_k$ belong to the standard dyadic grid.

As before, let $N(x)$ be the counting function over the selected intervals $J_k$ and  $N_I(x)$ be the counting function $N$ restricted to $J_k$ contained in interval $I$. 
We need the following analogue of Lemma~\ref{l.L1byLinfty}, and the rest of the proof for the $L^2$ portion (with respect to upper half of tents) is exactly the same as the $L^\infty$ portion in Section~\ref{ssub:Linfsize}.

\begin{lemma}\label{l.L1byLinfty-L2} Fix a dyadic interval $I$, $\alpha>0$, and integer $n \ge 0$.  Let $\chi_I(x)$ be as defined in \eqref{e.localization}. Then
$$\lambda \|N_I\|_{L^1(w)}^{1/q} \lesssim_{n,\alpha} \|N_I\|_{\infty}^\alpha \|f \chi_I^n\|_{L^q(w)}.$$
\end{lemma}

\proof
As before, we may assume without loss of generality that $N=N_I$ (freeing up the notation of interval $I$) and set $n=0$.

Let $S$ be the following square function
$$ S (F)(u) := \left(\sum_{k: J_k} \int_{T^*_k} |F(y,\eta,t)|^2 1_{|y-u|<t} \, dy d\eta \frac{dt}t \right)^{1/2}, \quad u\in \R.$$
Appealing to the doubling property of $w$ and the selection criteria for the tents $T_k$,
$$\lambda^2 \sum_k w(J_k) \lesssim \sum_{k} \int_{T^*_k} |P(f)(y,\eta,t)|^2 w(y-t,y+t) \, dyd\eta \frac{dt}{t} = \|S(P(f))\|_{L^2(w)}^2.$$
We note that if $(y,\eta,t)\in T(x_k,\xi_k,s_k)$ and $|y-u|<t$ then $u\in (y-t,y+t)\subset (x_k-s_k,x_k +s_k)$. Therefore 
$supp(S F) \subset \bigcup_{k: J_k} J_k$. Consequently, by an application of H\"older's inequality,
$$\lambda \|N\|_{L^1(w)}^{1/q} \lesssim \|S (P(f))\|_{L^q(w)}.$$

It remains to control the $L^q(w)$ norm of this square function.  Our main idea here is to cover $(y-t,y+t)$ using a  (shifted) dyadic interval  $I$   comparable to $(y-t,y+t)$ via the three grids trick and pass to (shifted) square functions over these grids.  More precisely, the interval $I$ will belong to either the classical  grid $\mathcal{D}_0$ or one of the shifted grids  $\mathcal{D}_1$, $\mathcal{D}_2$ mentioned prior. We may bound
$$S(F)(u) \le \sum_{j=0}^2 S_{j}(F)(u)$$
where $S_j(F)$ is a square function over grid $\mathcal{D}_j$.
Namely, for each grid $\mathcal{D}_j$  we may define for some $C>1$ (sufficiently large absolute constant)
$$S_j (F)(u) =  \Bigg(\sum_{k: J_k} \int_{T^*_k} |F(y,\eta,t)|^2 \sum_{I\in \mathcal{D}_j: t/C \le |I| \le Ct} 1_{I}(y) 1_{I}(u) \,  dy d\eta \frac{dt}t \Bigg)^{1/2}$$

Fix $j$. Standard estimates show 
$$\|S_j (P(f))\|_{L^q(w)} \lesssim \\|\big(S_j(P(f))\big)^{\sharp}\|_{L^q(w)}$$
where $(\cdot)^{\sharp}$ is  the dyadic sharp maximal function, with intervals from the grid $\mathcal{D}_j$. Now,
for each dyadic $J\in \mathcal{D}_j$, we then let
$$\alpha(u)= \Bigg(\sum_{k: I_k} \int_{T^*_k} |P(f)(y,\eta,t)|^2 \sum_{I\in D_j: C^{-1}t \le |I| \le Ct \,,\, J\subset I} 1_{I}(y) 1_I(u) \, dy d\eta \frac{dt}t \Bigg)^{1/2} .$$
Note that $\alpha(u)$ is constant over $u\in J$, which we now refer to as $\alpha_J$ (despite its dependence on $j$). 
For each $u\in J$,
\begin{align*}
    |S_j \big(P(f)\big)(u) - \alpha_J| &\le \Bigg(\sum_{k: J_k}  \int_{T^*_k} |P(f)(y,\eta,t)|^2 \sum_{I\in D_j: C^{-1}t \le |I| \le Ct \,,\, I\subset J} 1_{I}(y) 1_I(u) \, dy d\eta \frac{dt}t \Bigg)^{1/2} \\
        &\lesssim  \Bigg(\sum_{k: I_k}  \int_{T^*_k} |P(f)(y,\eta,t)|^2   1_{|y-u| = O( t)} 1_{t=O(|J|)} \, dy d\eta \frac{dt}t \Bigg)^{1/2}
\end{align*} 
thus using Lebesgue theory we have
$$\frac{1}{|J|} \int_J |S_j\big(P(f)\big)(u)-\alpha_J| \lesssim \Bigg(\frac 1 {|J|} \sum_{k: J_k} \int_{T^*_k} |P(f)(y,\eta,t)|^2 1_{(y-t,y+t)\subset CJ} \, dy d\eta dt \Bigg)^{1/2}$$
where $C>0$ is sufficiently large. Recall $(y-t,y+t)\subset CJ$ is equivalent to $0<t\le C|J|/2$ and $|y-c_J| < C|J|/2 - t$. 
By Remark~\ref{r.obs_enlarge}, $T^*_k \cap \{(y-t,y+t)\subset CJ\}$ is a subset of  tent $T'_k$ with top interval $I_k' =I_k \cap CJ$. 
Note the collection of partial tents  $T^*_k \cap \{(y-t,y+t)\subset CJ\}$ in $T'_k$  is still well separated with same separation constants as Claim~\ref{cl.SAlg_L2_WellSep}.

Let $f_{J} = f {\chi}_{CJ}^N$ for some large constant $N$. Using Lemma~\ref{l.wellsep2} with $|P(f)(y,\eta,t)| = |\< f_{J}, \chi_{CJ}^{-N}\phi_{y,\eta,t}\>|$ and $|P(f)(y,\eta,t)| \le \lambda$ for all points in $T^*_k \cap \{(y-t,y+t)\subset CJ\}$, we have
    \begin{align*}
        \int_J \big|S_j(P(f))(u)-\alpha_J\big| \, du  &\lesssim \|f_J\|_2 +  \big[\lambda  \|N\|_{L^1(CJ)}^{1/2}\big]^s \|f_J\|_2^{1-s} \\
        &\lesssim |J|^{1/2}\inf_{x\in J} M_2(f)(x) + |J|^{1/2} \lambda^s  \inf_{x\in J} \big[M(N)(x)\big]^{s/2} \inf_{x\in J} \big[M_2(f)(x)\big]^{1-s}
    \end{align*}
and so
\begin{equation}\label{e.sharpL2}
    \big(S_j(P(f))\big)^{\sharp}(x) \lesssim M_2(f)(x) + \lambda^s \big[ M(N)(x)\big]^{s/2} \big[M_2(f)(x)\big]^{1-s}.
\end{equation}
Combining \eqref{e.sharpL2} and the fact $w\in \A_{q/2}$, 
\begin{align*}
    \|S_j (P(f))\|_{L^q(w)} &\lesssim  \|f\|_{L^q(w)} +  \lambda^{s} \|N\|_{L^{q/2}(w)}^{s/2} \|f\|_{L^q(w)}^{1-s} \\
        &\lesssim \|f\|_{L^q(w)} + \lambda^s \|N\|_{\infty}^{(\frac{s}{q})(\frac{q}{2}-1)}   \|N\|_{L^1(w)}^{s/q}   \|f\|_{L^q(w)}^{1-s}
\end{align*}
Summing over $j=0,1,2$ we obtain
$$\lambda \|N\|_{L^1(w)}^{1/q} \lesssim \|f\|_{L^q(w)} + \lambda^s \|N\|_{\infty}^{(\frac{s}{q})(\frac{q}{2}-1)}   \|N\|_{L^1(w)}^{s/q}   \|f\|_{L^q(w)}^{1-s}$$
and we get by rearrangement
$$\lambda \|N\|_{L^1(w)}^{1/q} \lesssim  \|N\|_{\infty}^{\frac{s}{1-s}(\frac 12 -\frac 1 q)}\|f\|_{L^q(w)}.$$
We obtain the desired result for the $n=0$ case by choosing $s \in (0,1)$ sufficiently small. For the case $n \ge 1$, consider the mollified wave packet $\widetilde{\phi}_{y,\eta,t} = {\chi}_I^{-n} \phi_{y,\eta,t}$ and apply work above to $P(f)(y,\eta,t) = \< f {\chi}_I^n \, , \widetilde{\phi}_{y,\eta,t}\>$. 
\endproof

\bibliography{yendo}{}

\begin{thebibliography}{10}

\bibitem{bm2017}
C.~Benea and C.~Muscalu.
\newblock Sparse domination via the helicoidal method.
\newblock arXiv preprint arXiv:1707.05484, 2017.

\bibitem{bm2018}
C.~Benea and C.~Muscalu.
\newblock The helicoidal method.
\newblock In {\em Operator theory: themes and variations}, volume~20 of {\em
  Theta Ser. Adv. Math.}, pages 45--96. Theta, Bucharest, 2018.

\bibitem{carleson1966}
L.~Carleson.
\newblock On convergence and growth of partial sums of {F}ourier series.
\newblock {\em Acta Math.}, 116:135--157, 1966.

\bibitem{christ1988}
M.~Christ.
\newblock Weak type {$(1,1)$} bounds for rough operators.
\newblock {\em Ann. of Math. (2)}, 128(1):19--42, 1988.

\bibitem{cum2018}
D.~Cruz-Uribe and J.~M. Martell.
\newblock Limited range multilinear extrapolation with applications to the
  bilinear {H}ilbert transform, 2018.

\bibitem{cdpo2018}
A.~Culiuc, F.~Di~Plinio, and Y.~Ou.
\newblock Domination of multilinear singular integrals by positive sparse
  forms.
\newblock {\em J. Lond. Math. Soc. (2)}, 98(2):369--392, 2018.

\bibitem{dpdu2018}
F.~Di~Plinio, Y.~Q. Do, and G.~N. Uraltsev.
\newblock Positive sparse domination of variational {C}arleson operators.
\newblock {\em Ann. Sc. Norm. Super. Pisa Cl. Sci. (5)}, 18(4):1443--1458,
  2018.

\bibitem{dpo2018}
F.~Di~Plinio and Y.~Ou.
\newblock A modulation invariant {C}arleson embedding theorem outside local
  {$L^2$}.
\newblock {\em J. Anal. Math.}, 135(2):675--711, 2018.

\bibitem{dl2012sm}
Y.~Do and M.~Lacey.
\newblock Weighted bounds for variational {F}ourier series.
\newblock {\em Studia Math.}, 211(2):153--190, 2012.

\bibitem{dl2012jfa}
Y.~Do and M.~Lacey.
\newblock Weighted bounds for variational {W}alsh-{F}ourier series.
\newblock {\em J. Fourier Anal. Appl.}, 18(6):1318--1339, 2012.

\bibitem{dop2013}
Y.~Do, R.~Oberlin, and E.~A. Palsson.
\newblock Variational bounds for a dyadic model of the bilinear {H}ilbert
  transform.
\newblock {\em Illinois J. Math.}, 57(1):105--119, 2013.

\bibitem{dt2015}
Y.~Do and C.~Thiele.
\newblock {$L^p$} theory for outer measures and two themes of {L}ennart
  {C}arleson united.
\newblock {\em Bull. Amer. Math. Soc. (N.S.)}, 52(2):249--296, 2015.

\bibitem{fefferman1973}
C.~Fefferman.
\newblock Pointwise convergence of {F}ourier series.
\newblock {\em Ann. of Math. (2)}, 98:551--571, 1973.

\bibitem{hunt1968}
R.~A. Hunt.
\newblock On the convergence of {F}ourier series.
\newblock In {\em Orthogonal {E}xpansions and their {C}ontinuous {A}nalogues
  ({P}roc. {C}onf., {E}dwardsville, {I}ll., 1967)}, pages 235--255. Southern
  Illinois Univ. Press, Carbondale, Ill., 1968.

\bibitem{lt1997}
M.~Lacey and C.~Thiele.
\newblock {$L^p$} estimates on the bilinear {H}ilbert transform for
  {$2<p<\infty$}.
\newblock {\em Ann. of Math. (2)}, 146(3):693--724, 1997.

\bibitem{lt1999}
M.~Lacey and C.~Thiele.
\newblock On {C}alder\'{o}n's conjecture.
\newblock {\em Ann. of Math. (2)}, 149(2):475--496, 1999.

\bibitem{lt2000}
M.~Lacey and C.~Thiele.
\newblock A proof of boundedness of the {C}arleson operator.
\newblock {\em Math. Res. Lett.}, 7(4):361--370, 2000.

\bibitem{lerner2011}
A.~K. Lerner.
\newblock Sharp weighted norm inequalities for {L}ittlewood-{P}aley operators
  and singular integrals.
\newblock {\em Adv. Math.}, 226(5):3912--3926, 2011.

\bibitem{lerner2014}
A.~K. Lerner.
\newblock On sharp aperture-weighted estimates for square functions.
\newblock {\em J. Fourier Anal. Appl.}, 20(4):784--800, 2014.

\bibitem{li}
X.~Li.
\newblock personal communication.

\bibitem{mtt2004ma1}
C.~Muscalu, T.~Tao, and C.~Thiele.
\newblock {$L^p$} estimates for the biest. {I}. {T}he {W}alsh case.
\newblock {\em Math. Ann.}, 329(3):401--426, 2004.

\bibitem{mtt2004ma2}
C.~Muscalu, T.~Tao, and C.~Thiele.
\newblock {$L^p$} estimates for the biest. {II}. {T}he {F}ourier case.
\newblock {\em Math. Ann.}, 329(3):427--461, 2004.

\bibitem{osttw2012}
R.~Oberlin, A.~Seeger, T.~Tao, C.~Thiele, and J.~Wright.
\newblock A variation norm {C}arleson theorem.
\newblock {\em J. Eur. Math. Soc. (JEMS)}, 14(2):421--464, 2012.

\bibitem{thiele2006}
C.~Thiele.
\newblock {\em Wave packet analysis}, volume 105 of {\em CBMS Regional
  Conference Series in Mathematics}.
\newblock Published for the Conference Board of the Mathematical Sciences,
  Washington, DC; by the American Mathematical Society, Providence, RI, 2006.

\bibitem{ttv2015}
C.~Thiele, S.~Treil, and A.~Volberg.
\newblock Weighted martingale multipliers in the non-homogeneous setting and
  outer measure spaces.
\newblock {\em Adv. Math.}, 285:1155--1188, 2015.

\bibitem{uraltsev2016}
G.~Uraltsev.
\newblock Variational {C}arleson embeddings into the upper 3-space.
\newblock preprint ArXiv:1610.07657, 2016.

\end{thebibliography}
\bibliographystyle{abbrv}
\end{document}